\documentclass[11pt,leqno]{article}

\usepackage{fullpage}
\usepackage{amsthm,amsfonts,amssymb,amsmath,oldgerm}
\usepackage{graphicx}  
\usepackage{mathrsfs}
\usepackage{color}
\usepackage{tikz}
\usepackage{pgfplots}
\usetikzlibrary{arrows.meta}

\usepackage[hidelinks]{hyperref}

\newcommand{\R}{{\mathbb R}}\newcommand{\N}{{\mathbb N}}

\newcommand{\Z}{{\mathbb Z}}\newcommand{\C}{{\mathbb C}}
\newcommand{\sigfunk}{\Sigma}

\newcommand{\tildeL}{\widetilde{L}_{c^\ast}}
\newcommand{\tildeLh}{\widetilde{L}_{c^\ast,h}}

\newcommand{\tildeD}{\widetilde{\mathcal{D}}_{c^\ast}}
\newcommand{\tildeH}{\widetilde{\mathcal{H}}_{c^\ast}}
\newcommand{\tildeDh}{\widetilde{\mathcal{D}}_{h}}
\newcommand{\tildeHh}{\widetilde{\mathcal{H}}_{h}}
\newcommand{\tildeDn}{\widetilde{\mathcal{D}}_{n}}

\newcommand{\tildepin}{\widetilde{\pi}_{n}}
\newcommand{\tildepih}{\widetilde{\pi}_h}

\newcommand{\curlN}{\mathcal{N}}
\newcommand{\curlR}{\mathcal{R}}
\newcommand{\curlO}{\mathcal{O}}
\newcommand{\curlW}{\mathcal{W}}
\newcommand{\curlV}{\mathcal{V}}
\newcommand{\curlD}{\mathcal{D}}
\newcommand{\curlU}{\mathcal{U}}
\newcommand{\curlH}{\mathcal{H}}
\newcommand{\curlP}{\mathcal{P}}
\newcommand{\snorm}[2][]{\ensuremath{\left\vert#2\right\vert_{#1}}}
\newcommand{\norm}[2][]{\ensuremath{\left\|#2\right\|_{#1}}}

\let\epsilon\varepsilon
\let\theta\vartheta
\let\hat\widehat

\newtheorem{theorem}{Theorem}[section]\newtheorem{lemma}[theorem]{Lemma}

\newtheorem{proposition}[theorem]{Proposition}
\newtheorem{conjecture}[theorem]{Conjecture}
\newtheorem{corollary}[theorem]{Corollary}

\newtheorem{remark}[theorem]{Remark}

\title{Moving modulating pulse and front solutions of \\ permanent form in a  FPU model with
\\ nearest and next-to-nearest neighbor  interaction}

\author{Bastian Hilder\footnotemark[1] \footnotemark[2] \qquad Bj\"orn de Rijk\footnotemark[1] \footnotemark[3] \qquad Guido Schneider\footnotemark[1] \footnotemark[4]}

\date{\today}

\begin{document}

\maketitle

\renewcommand{\thefootnote}{\fnsymbol{footnote}}
\footnotetext[1]{Institut f\"ur Analysis, Dynamik und Modellierung, Universit\"at Stuttgart, Pfaffenwaldring 57, 70569 Stuttgart, Germany}
\footnotetext[2]{Current address: Centre for Mathematical Sciences, Lund University, PO Box 118, 221 00 Lund, Sweden; \texttt{bastian.hilder@math.lu.se}}
\footnotetext[3]{Current address: Karlsruhe Institute of Technology, Englerstra{\ss}e 2, 76131 Karlsruhe, Germany; \texttt{bjoern.de-rijk@kit.edu}}
\footnotetext[4]{\texttt{guido.schneider@mathematik.uni-stuttgart.de}}

\begin{abstract}
We consider a nonlinear chain of coupled oscillators, which is a direct generalization of the classical FPU lattice and exhibits, besides the usual nearest neighbor interaction, also next-to-nearest neighbor interaction. For the case of nearest neighbor attraction and next-to-nearest neighbor repulsion we prove that such a lattice admits, in contrast to the classical FPU model, moving modulating front solutions of permanent form, which have small converging tails at infinity and can be approximated by solitary wave solutions of the Nonlinear Schr\"odinger equation. When the associated potentials are even, then the proof yields moving modulating pulse solutions of permanent form, whose profiles are spatially localized. Our analysis employs the spatial dynamics approach as developed by Iooss and Kirchg\"assner. The relevant solutions are constructed on a five-dimensional center manifold and their persistence is guaranteed by reversibility arguments.
\end{abstract}


\section{Introduction}

The classical Fermi-Pasta-Ulam(-Tsingou) (FPU) system
\begin{equation}\label{FPUclassic}
\partial_t^2 q_n = \mathcal{W}'(q_{n+1}(t)-q_n(t)) - \mathcal{W}'(q_n(t) -
q_{n-1}(t)), \qquad n \in \Z,
\end{equation}
with potential
function  $\mathcal{W} \colon \R \to \R $,
was first
studied numerically by Fermi, Pasta, Ulam, and Tsingou~\cite{FPU1955}
in order to see how energy is
spread through a nonlinearly nearest neighbor coupled  oscillator chain.
They found rather
regular motion and no thermalization. This unexpected behavior has been explained
by Kruskal and Zabusky in~\cite{ZK1965}, where they
derived the Korteweg-de Vries (KdV) equation with its soliton dynamics as a formal approximation
of the FPU system. In detail, inserting the long-wave ansatz
$$
q_n(t) = \varepsilon^2 A \left(\varepsilon (n-ct),\varepsilon^3 t\right),
$$
into~\eqref{FPUclassic}, with small
perturbation parameter $ 0 < \varepsilon \ll 1 $ and
velocity $ c $, one finds that
the amplitude $ A(X,T) \in \mathbb{R} $ has to satisfy, at lowest order in $\varepsilon$, a
KdV equation
\begin{equation} \label{kdv}
\partial_T A = \nu'_1 \partial_X^3 A +  \nu'_2 A \partial_X A,
\end{equation}
with coefficients $ \nu'_1, \nu'_2 \in  \mathbb{R} $.
A rigorous proof that long waves in the classical FPU system~\eqref{FPUclassic}
can be  approximated by solutions to the KdV equation
on the natural time scale $ \mathcal{O}(1/\varepsilon^3) $
has been given in~\cite{SW00equa}, see~\cite{GMWZ} for further developments.

The KdV equation~\eqref{kdv} possesses solitary waves of permanent form
$$
A(X,T) = A_{sol,\tilde{c}}(X-\widetilde{c} T) = \gamma_1 \widetilde{c} \ \mathrm{sech}^2 \left(\gamma_2 \sqrt{\widetilde{c}}  (x - \widetilde{c}  t - a)\right),
$$
parameterized by $ \widetilde{c} , a\in \mathbb{R} $ and
with $ \gamma_1,\gamma_2 \in \mathbb{R} $ some constants depending on $ \nu'_1, \nu'_2 \in  \mathbb{R} $. It is then natural to ask whether~\eqref{FPUclassic} also possesses
moving (or traveling) pulse or front solutions of permanent form, i.e.~solutions of the form
\begin{equation} \label{vintro}
q_n(t) = v(n-ct), \qquad n \in \N,
\end{equation}
with profile function $v \colon \R \to \R$ converging to well-defined limits $\lim_{\xi \to \pm\infty} v(\xi)= v_\pm \in \R$ as $\xi \to \pm \infty$, where we have $v_+ = v_-$ in case of a pulse, and where $v_+$ and $v_-$ are not necessarily equal in case of a front. Moreover, we stress that, due to translation invariance of~\eqref{FPUclassic} and its invariance under the shift map $q \mapsto q + q_0$ with $q_0 \in \R$, solutions always arise in two-parameter families. The shift invariance in particular implies that, if a moving pulse solution of permanent form with $v_\pm \neq 0$ exists, then there exists a \emph{localized} moving pulse solution of permanent form~\eqref{vintro} with profile function $v \colon \R \to \R$ satisfying $\lim_{\xi \to \pm \infty} v(\xi) = 0$.

The question whether~\eqref{FPUclassic} admits moving pulses or fronts of permanent form is non-trivial and has first been answered positively in~\cite{FW94}
using variational methods.
In~\cite{Io00} such solutions
and their asymptotic KdV form for small amplitude
 are constructed via bifurcation theory, a spatial dynamics approach, and center manifold theory.
Here, the traveling wave solutions  solve  a scalar advance-delay differential equation, which can be interpreted as  a reversible infinite-dimensional differential equation with respect to $ \xi $.  We refer to~\cite{JS08} for an overview about important developments
about the
mathematical theory of exact solitary waves in this system.

So far,
all moving pulses or fronts of permanent form,
which have been found for~\eqref{FPUclassic} using
linear bifurcation theory
are of long-wave character for small amplitudes
and
can be approximated by solitary wave
solutions of the associated KdV equation.
It is the purpose of this paper to construct
moving pulses and fronts of small amplitude
for a slightly modified FPU model which are not of long wave form
and which are not approximately  given by  the solitary wave
solutions of the associated KdV equation, but are approximated by solitary wave
solutions of a Nonlinear Schr\"odinger (NLS) equation instead.

For this purpose we consider the FPU model
\begin{equation}\label{FPU}
    \partial_t^2 q_n =\mathcal{W}'_{1}\left( q_{n+1}-q_{n}\right) - \mathcal{W}'_{1}\left( q_{n}-q_{n-1}\right) +
    \mathcal{W}'_{2}\left( q_{n+2}-q_{n}\right)
    - \mathcal{W}'_{2}\left( q_{n}-q_{n-2}\right) ,
\end{equation}
where $\mathcal{W}_{1,2} \colon \R \to \R$ are smooth potential functions, which can be expanded as
\begin{align} \label{expWs}
\mathcal{W}'_{1}\left( r\right) = 5r+a_1r^2 + b_1r^3 + \mathcal{O}\big( r^{4}\big), \qquad \mathcal{W}'_{2}\left( r\right) =-r+a_2r^2 + b_2r^3 + \mathcal{O}\big( r^{4}\big),
\end{align}
with coefficients $a_{1,2},b_{1,2} \in \R$. Thus, for small displacements the nearest neighbor interaction is attracting, whereas the next-to-nearest neighbor interaction is repelling, see Figure~\eqref{fig1}.

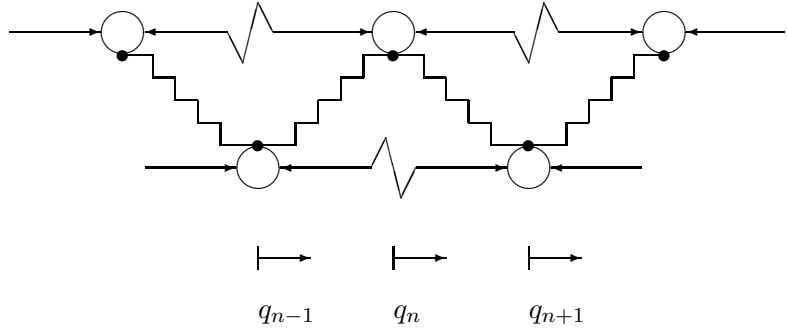
\begin{figure}[htbp] 
   \centering
    \setlength{\unitlength}{1cm}
   \begin{picture}(12, 5.5)


  \put(1.9,4.5){\line(1,0){0.4}}
   \put(2.3,4.5){\line(0,-1){0.3}}
  \put(2.3,4.2){\line(1,0){0.3}}
 \put(2.6,4.2){\line(0,-1){0.3}}
  \put(2.6,3.9){\line(1,0){0.3}}
   \put(2.9,3.9){\line(0,-1){0.3}}
  \put(2.9,3.6){\line(1,0){0.3}}
 \put(3.2,3.6){\line(0,-1){0.3}}
  \put(3.2,3.3){\line(1,0){0.5}}

 \put(5.5,4.5){\line(1,0){0.4}}
   \put(5.9,4.5){\line(0,-1){0.3}}
  \put(5.9,4.2){\line(1,0){0.3}}
 \put(6.2,4.2){\line(0,-1){0.3}}
  \put(6.2,3.9){\line(1,0){0.3}}
   \put(6.5,3.9){\line(0,-1){0.3}}
  \put(6.5,3.6){\line(1,0){0.3}}
 \put(6.8,3.6){\line(0,-1){0.3}}
  \put(6.8,3.3){\line(1,0){0.5}}

 \put(8.7,4.5){\line(1,0){0.4}}
   \put(8.7,4.2){\line(0,1){0.3}}
  \put(8.4,4.2){\line(1,0){0.3}}
 \put(8.4,3.9){\line(0,1){0.3}}
  \put(8.1,3.9){\line(1,0){0.3}}
   \put(8.1,3.6){\line(0,1){0.3}}
  \put(7.8,3.6){\line(1,0){0.3}}
 \put(7.8,3.3){\line(0,1){0.3}}
  \put(7.8,3.3){\line(-1,0){0.5}}

  \put(5.1,4.5){\line(1,0){0.4}}
   \put(5.1,4.2){\line(0,1){0.3}}
  \put(4.8,4.2){\line(1,0){0.3}}
 \put(4.8,3.9){\line(0,1){0.3}}
  \put(4.5,3.9){\line(1,0){0.3}}
   \put(4.5,3.6){\line(0,1){0.3}}
  \put(4.2,3.6){\line(1,0){0.3}}
 \put(4.2,3.3){\line(0,1){0.3}}
  \put(4.2,3.3){\line(-1,0){0.5}}

\put(3,3.7){\line(1,1){0.3}}

 \put(1.8,4.4){$ \bullet $}
 \put(5.4,4.4){$ \bullet $}
 \put(9,4.4){$ \bullet $}

 \put(3.6,3.2){$ \bullet $}
  \put(7.2,3.2){$ \bullet $}

 \put(1.9,4.8){\circle{0.6}}
 \put(3.7,3){\circle{0.6}}

 \put(5.5,4.8){\circle{0.6}}
 \put(7.3,3){\circle{0.6}}

 \put(9.1,4.8){\circle{0.6}}

 \put(3.7,1){$q_{n-1}$}
  \put(5.5,1){$q_{n}$}
  \put(7.3,1){$q_{n+1}$}

     \put(3.7,1.65){\line(0,1){0.3}}
   \put(3.7,1.8){\vector(1,0){0.7}}
     \put(5.5,1.65){\line(0,1){0.3}}
   \put(5.5,1.8){\vector(1,0){0.7}}
  \put(7.3,1.65){\line(0,1){0.3}}
   \put(7.3,1.8){\vector(1,0){0.7}}

   \put(2.2,3){\vector(1,0){1.2}}
   \put(5.8,3){\vector(1,0){1.2}}
     \put(5.2,3){\line(1,2){0.2}}
   \put(5.4,3.4){\line(1,-4){0.2}}
  \put(5.6,2.6){\line(1,2){0.2}}
     \put(5.2,3){\vector(-1,0){1.2}}
       \put(8.8,3){\vector(-1,0){1.2}}

           \put(0.4,4.8){\vector(1,0){1.2}}
       \put(3.9,4.8){\vector(1,0){1.3}}
     \put(3.3,4.8){\line(1,-2){0.2}}
   \put(3.5,4.4){\line(1,4){0.2}}
  \put(3.7,5.2){\line(1,-2){0.2}}
     \put(3.3,4.8){\vector(-1,0){1.1}}

       \put(10.7,4.8){\vector(-1,0){1.3}}
   \put(7.7,4.8){\vector(1,0){1.1}}
     \put(7.1,4.8){\line(1,-2){0.2}}
   \put(7.3,4.4){\line(1,4){0.2}}
  \put(7.5,5.2){\line(1,-2){0.2}}
     \put(7.1,4.8){\vector(-1,0){1.3}}

\end{picture}
    \caption{Sketch of a possible physical realisation of the oscillator chain modeled by~\eqref{FPU}. It consists of
    two lines of oscillators where in between each line the forces are repelling and across
    the lines the  forces are attracting.}
    \label{fig1}
\end{figure}

The linearized problem
\begin{equation}
\label{linFPU}
\partial_t^2 q_{n}=5\left( q_{n+1}-2q_{n}+q_{n-1}\right)
-\left( q_{n+2}-2q_{n}+q_{n-2}\right),
\end{equation}
is solved by
$$
q_{n}(t)=e^{i\left( kn-\omega t\right) },
$$
where the spatial and temporal wave numbers $k \in \R$ and $\omega \in \R$ are linked through the linear dispersion relation
\begin{equation}\label{disptemp}
\omega^{2} = 10\left( 1-\cos(k)\right) -2\left( 1-\cos(2k)\right),
\end{equation}
see Figure~\ref{fig2}.

It was first observed in~\cite{WAT} that the inclusion of repelling next-to-nearest neighbor interaction in~\eqref{FPU} leads to phenomena which are not exhibited by the classical FPU model~\eqref{FPUclassic}. Indeed, in contrast to the classical FPU model, one formally establishes in~\cite{WAT}, for a rescaled version of~\eqref{FPU} with purely quadratic potential $\mathcal{W}_2$, that moving pulse solutions of permanent form~\eqref{vintro} exist for wavespeeds $c$ slightly below the group velocity $|\omega'(0)| = 1$ of the wavenumber $k = 0$. This group velocity $|\omega'(0)|$ is also known as the \emph{speed of sound}. Traveling solutions of permanent form propagating with speed $c < |\omega'(0)|$ are then adequately called \emph{subsonic}, whereas moving solutions of permanent form with speed $c > |\omega'(0)|$ are called \emph{supersonic}.

Later, it was proved in~\cite{VENZ} using variational methods that~\eqref{FPU} admits, in contrast to~\eqref{FPUclassic}, slightly supersonic solutions with periodic profile functions. Further generalizations of~\eqref{FPU} allowing for fully nonlocal interaction, i.e.~interaction between \emph{all} oscillators, have been rigorously considered in~\cite{HERRM,PANK}. In~\cite{HERRM} one establishes slightly supersonic pulse solutions using methods from asymptotic analysis, whereas in~\cite{PANK} traveling solutions of permanent form with periodic or with localized profile functions propagating with arbitrary supersonic speeds are obtained using variational methods requiring certain monotonicity assumptions on the potentials.

Recently, it was suggested in~\cite{TRUV}, for the case of a purely quadratic potential $\mathcal{W}_2$, that~\eqref{FPU} admits a so-called velocity gap separating near-to-sonic pulse solutions, with wavespeeds near the speed of sound $|\omega'(0)|$, from strictly supersonic pulse solutions, with wavespeeds slightly above the maximum value of $\omega(k)/k$. For the linear dispersion relation~\eqref{disptemp} this gap is manifested by the inequality $1 = |\omega'(0)| < \sup\{\omega(k)/k : k > 0\}$ as can be observed from Figure~\ref{fig2} (compare also the upper middle panels of the upcoming Figures~\ref{fig3} and~\ref{fig3sec5}). For the specific case of a piecewise linear $\mathcal{W}_1'$ an explicit series expansion of such a strictly supersonic pulse solution was obtained in~\cite{TRUV} and a formal NLS approximation was established.

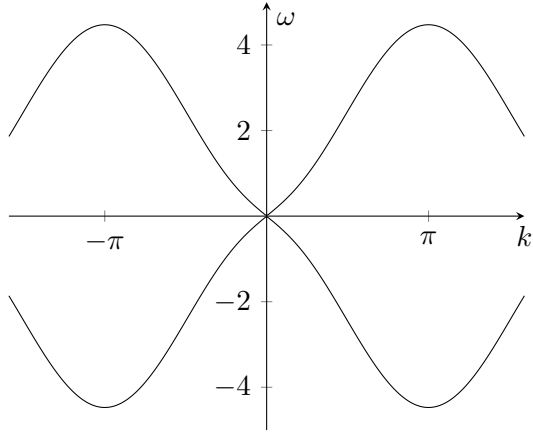
\begin{figure}
\centering
\begin{tikzpicture}[xscale=1,yscale=1]
\begin{axis}
[
	xmin=-5, xmax=5,
	ymin=-5, ymax=5,
	axis lines=center,
	xlabel = $k$,
	ylabel = $\omega$,
	x label style = {anchor=north},
	xtick = {-3.141,3.141},
	xticklabels = {$-\pi$,$\pi$},
    ]
 \addplot [domain =-5:5 ,smooth, samples = 244 ]{sqrt(10*(1-cos(180*x/3.14))- 2*(1-cos(180*2*x/3.14)))};
 \addplot [domain =-5:5 ,smooth, samples = 244 ]{-sqrt(10*(1-cos(180*x/3.14))- 2*(1-cos(180*2*x/3.14)))};
\end{axis}
\end{tikzpicture}

    \caption{Plot of the linear dispersion relation~\eqref{disptemp} in the $(k,\omega)$-plane, which is $2\pi$-periodic with respect to $k$.}
    \label{fig2}
\end{figure}

In this paper we rigorously establish such strictly supersonic pulse solutions for the FPU model~\eqref{FPU}, which are approximated by solitary waves of an associated NLS equation, allowing for a large class of potentials $\mathcal{W}_1$ and $\mathcal{W}_2$. Thus, we are interested in
moving pulse and front solutions to~\eqref{FPU} of permanent form, that is, solutions to~\eqref{FPU} of the form~\eqref{vintro}, whose profile function $v \colon \R \to \R$ satisfies $ \lim_{\xi \to \pm\infty} v(\xi)= v_\pm$ for some $v_\pm \in \R$, see Figure~\ref{fig:solutions}.
In addition, we obtain solutions, which have an oscillatory character and can, for small amplitude, be approximated by solutions of an associated
NLS equation. Thus, the solutions we are going to construct are of the form
\begin{align} \label{e:NLSansatz}
q_{n}( t) = \varepsilon A\left( \varepsilon \left( n-c_{g}t\right) ,\varepsilon ^{2}t\right) e^{i\left( k_{0}n-\omega _{0}t\right) }+c.c.,
\end{align}
with small
perturbation parameter $ 0 < \varepsilon \ll 1 $, amplitude $ A(X,T) \in \mathbb{C} $, and
group velocity $ c_g $. The spatial and temporal wave numbers $ k_0 > 0$ and $ \omega_0 > 0 $ are related through the linear dispersion relation~\eqref{disptemp}.
Inserting the ansatz~\eqref{e:NLSansatz} into~\eqref{FPU} shows that, after an expansion with respect to $\varepsilon$,  $ A $ has to satisfy
a NLS equation
\begin{equation}  \label{NLS}
\partial_T A = i \nu_1 \partial_X^2 A + i \nu_2 A |A|^2,
\end{equation}
with coefficients $ \nu_1, \nu_2 \in  \mathbb{R} $.
A rigorous proof that waves of the form~\eqref{e:NLSansatz} in the classical FPU system~\eqref{FPU}
can be  approximated by solutions to the NLS equation
on the natural time scale $ \mathcal{O}(1/\varepsilon^2) $
has been given in~\cite{Schn10AA}. Moreover, NLS approximation results for polyatomic FPU chains can be found in~\cite{CBCPS}.

In case $ \nu_1 \nu_2 > 0 $ the
NLS equation~\eqref{NLS} possesses spatially localized time-periodic solutions
\begin{equation} \label{ahom}
	A(X,T) = A_{\text{hom},\gamma}(X) e^{i \gamma T}, \quad A_{\text{hom},\gamma}(X) = \sqrt{\dfrac{2\gamma}{\nu_2}} \dfrac{1}{\operatorname{cosh}(\sqrt{\gamma/\nu_1}X)},
\end{equation}
with $ \lim_{|X| \to \infty} A_{\text{hom},\gamma}(X) = 0 $ for  every $ \gamma \in \mathbb{R}$
with $  \nu_1 \gamma > 0 $.
These so called breather solutions lead to
\emph{moving modulating pulse and front solutions} of the FPU system, namely
\begin{equation}  \label{movingpulse}
q_{n}( t) =
\varepsilon A_{hom,\gamma}\left( \varepsilon \left( n-c_g t\right) \right) e^{i\left( k_{0}(n-c_pt)\right) }+c.c. + \mathcal{O}(\varepsilon^2),
\end{equation}
where $ c_p = \omega_0/k_0$ denotes the phase velocity of the
underlying carrier wave.
Modulating pulse and front solutions are characterized by
$$
v_{mp}(\xi,p) = v_{mp}(\xi,p+2 \pi), \quad \textrm{with} \quad
\lim_{\xi \to \pm\infty} v_{mp}(\xi,p) = v_\pm ,
$$
where $ \xi \in \mathbb{R} $ is the variable for the envelope, $ p \in \mathbb{R} /(2 \pi \mathbb{Z} ) $ is  the variable for the underlying carrier wave, and the limits $v_\pm \in \R$ are equal in case of a pulse.
Modulating solutions $  v_{mp}(\xi,p) $ are time-periodic in a frame
co-moving
with the envelope, and, thus, are often written as $
v_{br}(\xi,t) =  v_{mp}(\xi,p)
$.
If $ v_{mp}(\xi,p) $ does not converge for $\xi \to \pm\infty $, but possesses
small oscillatory tails, then $ v_{mp}(\xi,p) $ is a so-called \emph{generalized modulating pulse (or front) solution}.
In~\eqref{movingpulse}, we have
\begin{equation} \label{mopu}
v_{mp}(\xi,p) = \varepsilon A_{hom,\gamma}(\xi) e^{ip}+c.c. + \mathcal{O}(\varepsilon^2),
\end{equation}
with $ \xi = n-c_g t $ and $ p = k_{0}(\xi-(c_p-c_g) t) $.
Such generalized modulating pulse solutions
have been constructed for
Klein-Gordon lattices
\begin{equation} \label{dKG}
\ddot{q}_{n}=-V'(q_n) + \left( q_{n+1}-2q_{n}+q_{n-1}\right),
\end{equation}
 in~\cite{jamesSire05} with localized potential $ V(q_n)  =  \frac12 q_n^2 + \mathcal{O}(q_n^3)$ resulting in $ \omega(0) > 0 $ (whereas we have $\omega(0) = 0$, see Figure~\ref{fig2}).

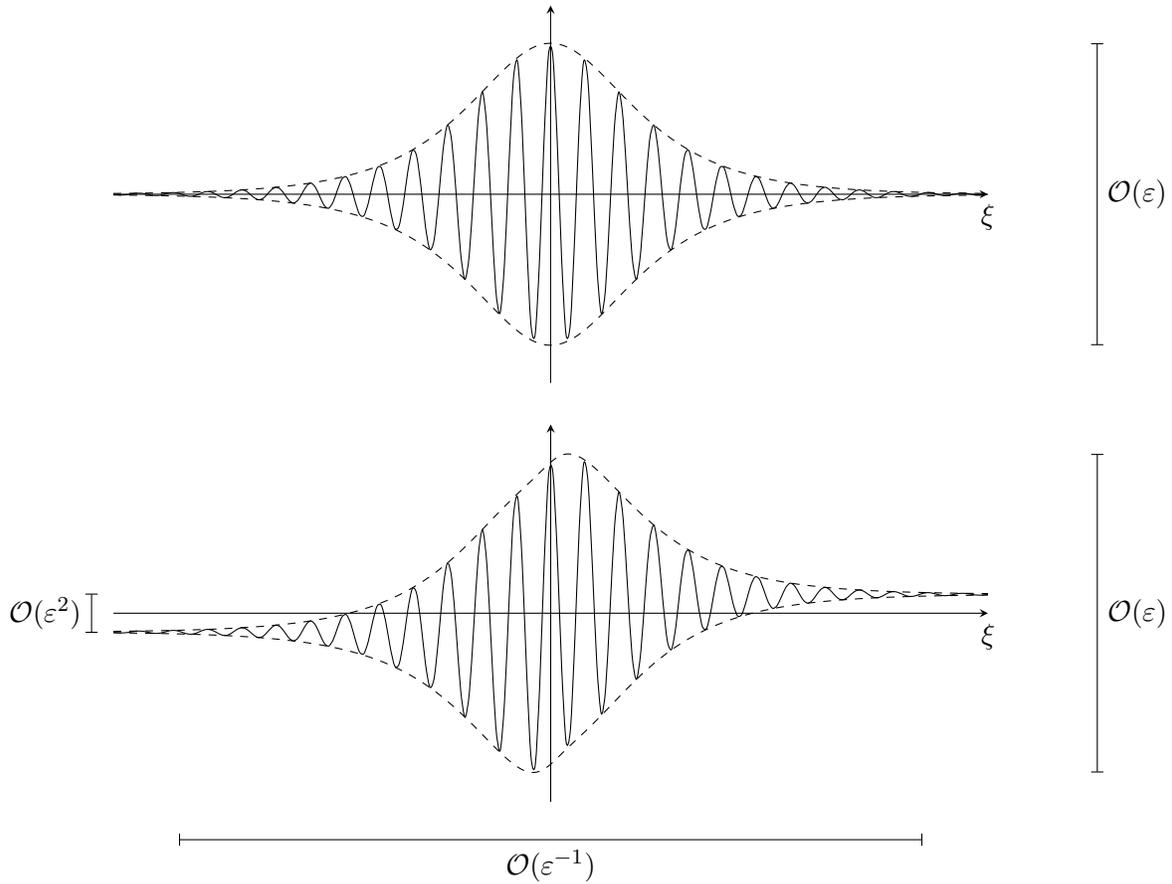
\begin{figure}
	\centering
	\hspace*{1.382cm}
	\begin{tikzpicture}
		\begin{axis}
 			[
 				width = 0.8*\textwidth,
 				height = 0.4*\textwidth,
 				xmin=-20, xmax=20,
				ymin=-2.5, ymax=2.5,
				axis lines=center,
				ticks=none,
				xlabel = $\xi$,
				x label style = {anchor = north},
				clip = false,
			]
 			\addplot [domain =-20:20 ,smooth, samples = 200]{2*cos(4*x*(180/3.141))/cosh(0.3*x)};
 			\addplot[domain = -20:20, smooth, samples = 200, dashed] {2/cosh(0.3*x)};
 			\addplot[domain = -20:20, smooth, samples = 200, dashed] {-2/cosh(0.3*x)};
 		
 			
 			\draw[|-|] (axis cs:25,2) -- (axis cs:25,-2) node[midway,right]{$\curlO(\varepsilon)$};
 		\end{axis}
 	\end{tikzpicture}
 	
 	\vspace*{0.5cm}
 	
	\begin{tikzpicture}
		\begin{axis}
 			[
 				width = 0.8*\textwidth,
 				height = 0.4*\textwidth,
 				xmin=-20, xmax=20,
				ymin=-2.5, ymax=2.5,
				axis lines=center,
				ticks=none,
				xlabel = $\xi$,
				x label style = {anchor = north},
				clip = false,
 			]
 			\addplot [domain =-20:20 ,smooth, samples = 200]{2*cos(4*x*(180/3.141))/cosh(0.3*x) + 0.25*tanh(x)};
 			\addplot[domain = -20:20, smooth, samples = 200, dashed] {2/cosh(0.3*x) + 0.25*tanh(x)};
 			\addplot[domain = -20:20, smooth, samples = 200, dashed] {-2/cosh(0.3*x) + 0.25*tanh(x)};
 		
 			
 			\draw[|-|] (axis cs:-21,0.26) -- (axis cs:-21,-0.26) node[midway,left]{$\curlO(\varepsilon^2)$};
 			\draw[|-|] (axis cs:25,2.11) -- (axis cs:25,-2.11) node[midway,right]{$\curlO(\varepsilon)$};
 			\draw[|-|] (axis cs:-17,-3) -- (axis cs:17,-3) node[midway,below]{$\curlO(\varepsilon^{-1})$};
 		\end{axis}
	\end{tikzpicture}
	\caption{Schematic depiction of the profile functions of a localized modulating pulse (top) and of a modulating front (bottom) in the FPU model~\eqref{FPU}, which propagate with strictly supersonic speed $c > c^\ast$ and are constructed in Theorem~\ref{theo:mainresult}. For both solutions, the amplitude is of order $\varepsilon = \sqrt{c-c^\ast}$ on a length scale of order $\varepsilon^{-1}$. Furthermore, both solutions have limits $\tau_\pm$ at $\pm\infty$ satisfying $\snorm{\tau_+ - \tau_-} = \curlO(\varepsilon^2)$.}
	\label{fig:solutions}
\end{figure}

Here, we are interested in the situation where the moving modulating pulse or front solutions are of permanent
form, i.e., when $ c_g(k_0) = c_p(k_0)$ for a  $ k_0 \neq  0 $. Moreover, we look for solutions which converge
as $ \xi \to \pm\infty $.
To our knowledge for lattice differential equations such solutions have not been constructed before.
Either moving solutions of permanent
form which are of KdV-type for small amplitudes have been found, or,
if they are of NLS form for small amplitudes, they have non-converging, small oscillatory tails at $\pm \infty$.
It is the purpose of this paper to show that there is
a relatively simple FPU system for which such moving modulating pulse of fronts solutions of permanent form exist, namely~\eqref{FPU}.

Our main result is as follows.

\begin{theorem} \label{theo:mainresult}
There exist a speed $c^\ast > 1$, a wave number $k_0 > 0$ and an open set $\curlP \subset \R^4$ containing $\R^2 \times \R_+^2$ such that for all coefficients $(a_1,a_2,b_1,b_2) \in \curlP$ of the potentials $\mathcal{W}_{1,2}$ in~\eqref{expWs} there exist constants $c_0,C > 0$ such that the following holds.
For all $c \in  (c^*, c^*+ c_0)$ the system~\eqref{FPU} possesses a moving front solution of permanent form~\eqref{vintro} with an amplitude of order $\varepsilon := \sqrt{c-c^\ast}$, where the profile function $v \colon \R \to \R$ is smooth and has limits at $\xi \rightarrow \pm\infty$ satisfying
$$ \lim_{\xi \to \pm\infty} v(\xi) = \curlO(\varepsilon^2).$$
Furthermore, there exist $\varepsilon$-independent parameters $\nu_1 > 0$, $\nu_2 > 0$ for~\eqref{NLS} and $\gamma > 0$ for~\eqref{ahom} such that $v$ enjoys the estimate
\begin{align}
	\sup _{\xi\in \mathbb{R} }\left| v\left( \xi\right) -\left( \varepsilon A_{\text{hom},\gamma}\left( \varepsilon \xi\right) e^{ik_{0}\xi}+c.c.\right) \right| \leq C\varepsilon^{2} \snorm{\operatorname{log}(\varepsilon)},
	\label{eq:NLSestimate}
\end{align}
where $ A(X,T) = A_{\text{hom},\gamma}(X) e^{i\gamma T}$ is the time-periodic solution to the NLS equation~\eqref{NLS} introduced in~\eqref{ahom}.
\end{theorem}

\begin{remark}\label{rem:explicitQuant}
	We point out that the quantities $c^\ast$, $k_0$, $\curlP$, $\nu_{1,2}$ and $\gamma$ in Theorem~\ref{theo:mainresult} can be given explicitly.
	The critical velocity $c^\ast$ and the wave number $k_0$ are determined by the linear dispersion relation $\sigfunk(\lambda;c) = 0$, see~\eqref{eq:linDispersion} and Lemma~\ref{lem:centralSpec}.
	In particular, Lemma~\ref{lem:centralSpec} implies that $c^\ast$ is indeed the maximum of $k \mapsto \omega(k)/k$ and that the maximum is attained at wave number $k_0$.
	Finally, the set $\curlP$ of admissible coefficients in the potentials is determined by~\eqref{e:signcond} and finally, the parameters $\nu_{1,2}$ and $\gamma$ for the associated NLS-equation are given in~\eqref{eq:nlsParams}.
\end{remark}

By making additional symmetry assumptions on the potentials $\curlW_1,\curlW_2$ in~\eqref{FPU} we can guarantee that the profile function $v$ in Theorem~\ref{theo:mainresult} is a true pulse solution.
More precisely, we prove the following result.

\begin{theorem}\label{thm:trueHomoclinic}
	Let the conditions of Theorem~\ref{theo:mainresult} be satisfied and additionally assume that $\curlW_1,\curlW_2$ are symmetric potentials, i.e.~it holds $\curlW_j(r) = \curlW_j(-r)$ for each $r \in \R$ and $j = 1,2$.
	Then, the profile function $v \colon \R \to \R$ from Theorem~\ref{theo:mainresult} satisfies
	\begin{align*}
		\lim_{\xi \rightarrow -\infty} v(\xi) = \lim_{\xi \rightarrow \infty} v(\xi).
	\end{align*}
    In particular, there exists a localized moving pulse solution of permanent form~\eqref{vintro} to~\eqref{FPU} with profile function $v \colon \R \to \R$ satisfying $\lim_{\xi \to \pm \infty} v(\xi) = 0$.
\end{theorem}

\begin{remark}
	We note that Theorem~\ref{thm:trueHomoclinic} holds for a large class of potentials.
	In particular, every potential of the form
	\begin{align*}
		\curlW_1(r) = 5r^2 + b_1 r^4 + \curlO(r^6), \quad \curlW_2(r) = -r^2 + b_2 r^4 + \curlO(r^6),
	\end{align*}
	with symmetric higher order terms and $b_1,b_2 > 0$ satisfies the assumptions of Theorem~\ref{thm:trueHomoclinic}.
\end{remark}

We emphasize that the moving solutions of permanent form obtained in Theorems~\ref{theo:mainresult} and~\ref{thm:trueHomoclinic} are strictly supersonic, propagating with wavespeeds $c > c^\ast$, strictly away from the group velocity $|\omega'(0)| = 1$ of the wavenumber $k = 0$, see Remark \ref{rem:explicitQuant}. They are therefore likely to correspond to the strictly supersonic solutions formally obtained in~\cite[Section~8]{TRUV} through NLS approximations for the case~\eqref{FPU} has a purely quadratic potential $\mathcal{W}_2$. As far as we are aware, the only other mathematical work establishing the existence of strictly supersonic solutions to~\eqref{FPU} is~\cite{PANK}. Although the framework in~\cite{PANK} is more general than ours allowing for fully nonlocal interaction, the potentials need to satisfy certain monotonicity assumptions due to the use of variational methods, which is not necessary in our setting. In addition, the focus in~\cite{PANK} is not on the wavespeed regime $0 < c - c^\ast \ll 1$, but rather on \emph{arbitrary} supersonic wavespeeds. Consequently, an NLS-type estimate as in Theorem~\ref{theo:mainresult} is naturally not obtained in~\cite{PANK}.

The proof of Theorem~\ref{theo:mainresult} employs the seminal approach as developed by Iooss and Kirchg\"assner in~\cite{Io00,KiIo00}, which relies on spatial dynamics, center manifold reduction and bifurcation theory. We refer to~\S\ref{sec:approach} for a further outline of this approach and the specifics of the current application.

For PDEs (in cylindrical domains) the construction of solitary waves or pulse solutions
by adopting a spatial dynamics approach and using center manifold theory goes back to~\cite{Ki82}. In the following years the method
has been employed for the construction of solitary
water waves, see~\cite{Groves,KiIo92} and~\cite{HaIo11} for an overview.
In addition, it has been used for the construction of generalized moving modulating
pulse solutions, cf.~\cite{GS01,GS08}, which do not converge for $\xi \to \pm\infty $, but have small oscillatory tails.

We discuss some results in the literature which are mathematically strongly related to the present paper. We do refrain from giving a complete overview of existence results for moving modulating pulse (or front) solutions or moving pulses or fronts of permanent form in NLS, FPU, or Klein-Gordon lattices, which are obtained via bifurcation theory, spatial dynamics, and/or center manifold theory. Instead, we refer to~\cite{JS08} for an overview.

First, we would like to mention~\cite{KiIo00}, where small-amplitude traveling waves of permanent form~\eqref{vintro}, which are approximately given by the NLS solitary waves, have been constructed in Klein-Gordon lattices. We do emphasize that, in contrast to the ones constructed in this paper, these waves have non-converging, small oscillatory tails. Moreover, in~\cite{jamesSire05} generalized modulating pulse solutions $ q_n(t) = v(n-ct,t) =  v(n-ct,t+T) $ have been constructed in Klein-Gordon lattices. Finally, in~\cite{IoossJames05} special modulating pulse solutions satisfying $ q_n(t)= q_{n+N}(t-T) $ have been established in FPU lattices. We emphasize that a linear bifurcation analysis yields that the pulse solutions constructed in this paper, i.e.,~moving modulating pulse solutions of permanent form~\eqref{vintro} which converge for $ \xi \to \pm\infty $ and which, for small amplitude, are approximately given by NLS solitons, cannot be constructed for the classical FPU problem~\eqref{FPUclassic} by the method in use, cf.~\S\ref{sec5}.

\begin{remark}{\rm
The solutions constructed in this present paper correspond
to the ones constructed in~\cite{KiIo92} for the water wave problem.
The solutions consist of a pulse like
envelope moving with a group velocity
$c_g$
modulating an underlying carrier wave moving
with the same velocity $ c_p = c_g$.
For the water wave problem the construction of moving
modulating pulses when $ c_g \neq c_p $ is still an open problem.
The similarities between the water wave problem
and the FPU problem can be seen by reproducing  Figure~\ref{fig3} for the
dispersion relation of the water problem
$$
\omega^2 = (k + \sigma k^3)\tanh(k),
$$
where $ \sigma \geq 0 $ is the surface tension parameter.
}
\end{remark}

\subsection{Approach and plan of paper} \label{sec:approach}

We adopt a spatial dynamics approach~\cite{Io00,KiIo00} to prove our main result, Theorem~\ref{theo:mainresult}. That is, we observe that if $q_n(t)$ is a solution to~\eqref{FPU} of permanent form~\eqref{vintro} with profile $v$, then $v$ satisfies the advance-delay differential equation
\begin{align}\label{spatdyn}
\begin{split}
c^{2}v''(\xi)&= \mathcal{W}^{'}_{1}\left( v\left( \xi+1\right) -v\left( \xi\right)
\right) - \mathcal{W}^{'}_{1}\left( v\left( \xi\right) -v\left( \xi-1\right) \right) \\&\qquad
+ \, \mathcal{W}^{'}_{2}\left( v\left( \xi+2\right) -v\left( \xi\right) \right)
-\mathcal{W}^{'}_{2}\left( v\left( \xi\right) -v\left( \xi-2\right) \right),
\end{split}
\end{align}
which can be regarded as an evolutionary system with respect to $ \xi $ and, thus, we call it the \emph{spatial dynamics formulation}. We note that~\eqref{spatdyn} admits a reversible symmetry and conserves a first integral. To prove Theorem~\ref{theo:mainresult}, we construct a homoclinic (pulses) or heteroclinic (fronts) solution to~\eqref{spatdyn}.

For this purpose, we proceed as in~\cite{IoossJames05}, cf.~Remark~\ref{rem:differences}. We start our analysis in~\S\ref{sec2} by first writing the advance-delay differential equation~\eqref{spatdyn} as a classical dynamical system, which inherits its reversible symmetry and its conservation of the first integral. Subsequently, we study the spectrum of the linearization in~\S\ref{sec:specAna}. The linearized version of~\eqref{spatdyn},
\begin{align*}
c^{2}v''\left( \xi\right) & = 5\left( v\left( \xi+1\right) -2v\left( \xi\right) +v\left( \xi-1\right) \right) - \left( v\left( \xi+2\right) -2v\left( \xi\right) +v\left( \xi-2\right) \right) ,
\end{align*}
is solved by $ v(\xi) =e^{\lambda \xi}$. The neutral\footnote{In the literature, these are also referred to as central eigenvalues.} eigenvalues
$ \lambda = ik $
thus satisfy $c^2k^2=\omega(k)^2$, respectively $ \omega(k) = \pm ck$, where the temporal wave number $\omega(k)$ is defined through the linear dispersion relation~\eqref{disptemp}. We show that there exists a unique value $c = c^\ast$ such that the spectrum possesses precisely three double eigenvalues on the imaginary axis, see Figure~\ref{fig3}. These neutral eigenvalues would allow for reduction of the system to a six-dimensional center manifold for $c \approx c^\ast$. However, before applying the center manifold reduction, we factor out the eigenvalue at the origin in~\S\ref{sec:trans}, which is related to translational invariance of equation~\eqref{spatdyn}, leading to a simpler, five-dimensional center manifold. The application of the center manifold theorem and the verification of the associated spectral and optimal regularity conditions can be found in~\S\ref{sec3}.

The obtained reduced system on the five-dimensional center manifold is then analyzed in~\S\ref{sec4}. In~\S\ref{sec41} we first reduce one dimension further by noting that~\eqref{spatdyn} conserves a first integral. The linearization of this first integral is directly related to the remaining $0$-eigenvalue and its conservation implies that the dynamics of~\eqref{spatdyn} on the associated eigenspace is constant up to nonlinear effects. Using a near identity change of variables, the almost constant dynamics on this neutral eigenspace can be captured by introducing an additional system parameter, i.e.~the value of the first integral. Subsequently, we study the resulting four-dimensional system, which governs the dynamics on the center manifold associated with the off-zero neutral eigenvalues. To simplify the analysis, the system is brought into its normal form in~\S\ref{sec42} and relevant coefficients of the normal form expansion are computed in~\S\ref{sec43}. The value of these coefficients yields explicit homoclinic solutions to the truncated normal form for $c - c^\ast > 0$ sufficiently small. In~\S\ref{sec44} we argue that these homoclinic solutions persist when reintroducing the higher-order terms to the normal form, where we exploit the reversibility of the system. Finally, in~\S\ref{sec:proofMainThm}, we prove that this yields the existence of homoclinic solutions to the reduced system on the center manifold and thus, via the center manifold theorem, implies the existence of pulse and front solutions to the full spatial dynamics formulation~\eqref{spatdyn}. This proves our main result, Theorem~\ref{theo:mainresult}.

In the final discussion in~\S\ref{sec5} we explain which other kinds of generalized modulating pulse solutions are exhibited by~\eqref{FPU}.

\begin{remark}\label{rem:differences}
{\upshape
Although the nonlinear oscillator chain under consideration in~\cite{IoossJames05} is different to ours and does not exhibit next-to-nearest neighbor interaction, their analysis applies in large lines to our situation. The reason for this is that the critical spectrum of the linearization of the spatial dynamics formulation is similar, exhibiting a double zero mode and two complex conjugated double off-zero modes. Therefore, the structure of the center manifold in~\cite{IoossJames05} is comparable to ours and the analysis in~\cite{IoossJames05} of the reduced equations can be transferred. However, the next-to-nearest neighbor effects yield a different outcome when converting back to the full problem. Indeed, in~\cite{IoossJames05} the obtained solutions have oscillatory non-converging tails, which contrasts with the front and pulse solutions in Theorem~\ref{theo:mainresult} that have converging tails, cf.~Figure~\ref{fig:solutions}.
}\end{remark}

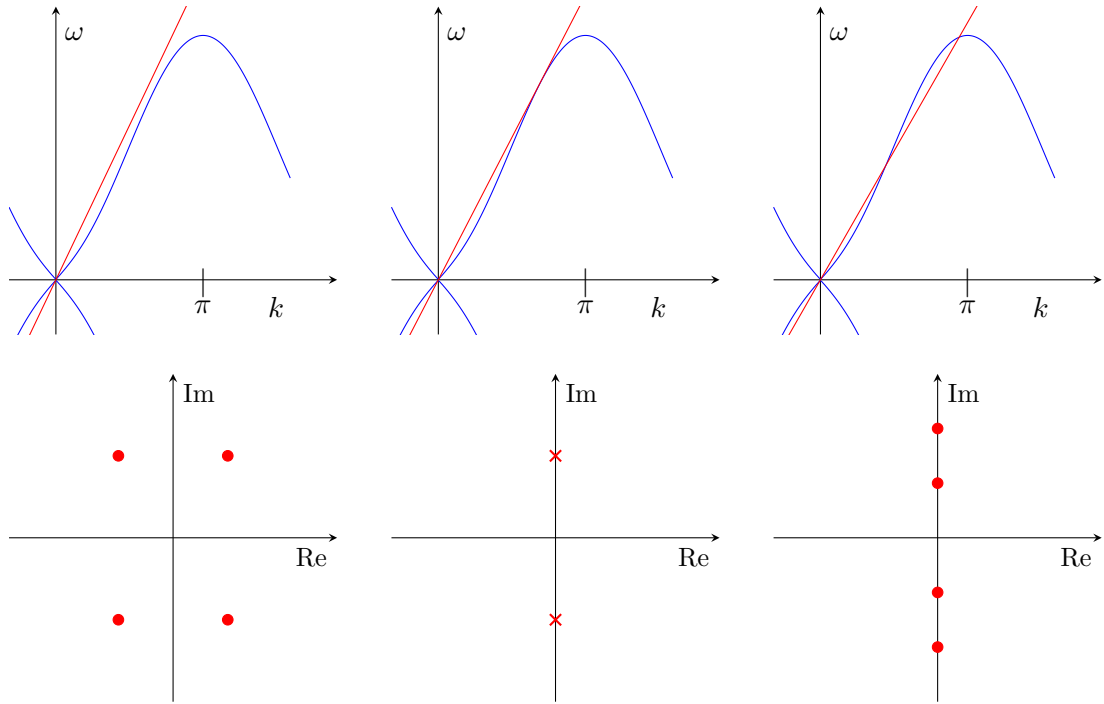
\begin{figure}
    \centering
    \begin{minipage}{0.3\textwidth}
    	\begin{tikzpicture}
    		\begin{axis}
				[
					width = 1.2\textwidth,
					height = 1.2\textwidth,
					xmin=-1, xmax=6,
					ymin=-1, ymax=5,
					axis lines=center,
					ticks=none,
    			]
 				\addplot [domain =-1:5 ,smooth, samples = 244,blue ]{sqrt(10*(1-cos(180*x/3.14))- 2*(1-cos(180*2*x/3.14)))};
 				\addplot [domain =-1:5 ,smooth, samples = 244,blue ]{-sqrt(10*(1-cos(180*x/3.14))- 2*(1-cos(180*2*x/3.14)))};
  				\addplot [domain =-1:5 ,smooth, samples = 244,red ]{1.8*x};

 				\node[black] at (axis cs:4.7,-0.5) {$k$};
 				\node[black] at (axis cs:3.14,-0.05) {$|$};
  				\node[black] at (axis cs:3.14,-0.5) {$\pi$};

  				\node[black] at (axis cs:0.4,4.5) {$\omega$};
  			\end{axis}
  		\end{tikzpicture}
  		\newline
		\newline
  		\begin{tikzpicture}
  			\begin{axis}
  				[
  					width = 1.2\textwidth,
  					height = 1.2\textwidth,
  					xmin = -3, xmax = 3,
  					ymin = -3, ymax = 3,
  					axis lines = center,
  					ticks = none,
  					xlabel = {\small$\operatorname{Re}$},
  					ylabel = {\small$\operatorname{Im}$},
  					x label style = {anchor = north east}
  				]
  				\addplot[only marks, red] table{
  					-1 1.5
  					-1 -1.5
  					1 1.5
  					1 -1.5
  				};
  			\end{axis}
  		\end{tikzpicture}
    \end{minipage}
    \begin{minipage}{0.3\textwidth}
    	\begin{tikzpicture}
    		\begin{axis}
				[
					width = 1.2\textwidth,
					height = 1.2\textwidth,
					xmin=-1, xmax=6,
					ymin=-1, ymax=5,
					axis lines=center,
					ticks=none,
    			]
 				\addplot [domain =-1:5 ,smooth, samples = 244,blue ]{sqrt(10*(1-cos(180*x/3.14))- 2*(1-cos(180*2*x/3.14)))};
 				\addplot [domain =-1:5 ,smooth, samples = 244,blue ]{-sqrt(10*(1-cos(180*x/3.14))- 2*(1-cos(180*2*x/3.14)))};
  				\addplot [domain =-1:5 ,smooth, samples = 244,red ]{1.66*x};

 				\node[black] at (axis cs:4.7,-0.5) {$k$};
 				\node[black] at (axis cs:3.14,-0.05) {$|$};
  				\node[black] at (axis cs:3.14,-0.5) {$\pi$};

  				\node[black] at (axis cs:0.4,4.5) {$\omega$};
  			\end{axis}
		\end{tikzpicture}
		\newline
		\newline
  		\begin{tikzpicture}
  			\begin{axis}
  				[
  					width = 1.2\textwidth,
  					height = 1.2\textwidth,
  					xmin = -3, xmax = 3,
  					ymin = -3, ymax = 3,
  					axis lines = center,
  					ticks = none,
  					xlabel = {\small$\operatorname{Re}$},
  					ylabel = {\small$\operatorname{Im}$},
  					x label style = {anchor = north east}
  				]
  				\addplot[only marks, mark=x, red, mark size = 3pt, thick] table{
  					0 1.5
  					0 -1.5
  				};
  			\end{axis}
  		\end{tikzpicture}
    \end{minipage}
    \begin{minipage}{0.3\textwidth}
    	\begin{tikzpicture}
    		\begin{axis}
				[
					width = 1.2\textwidth,
					height = 1.2\textwidth,
					xmin=-1, xmax=6,
					ymin=-1, ymax=5,
					axis lines=center,
					ticks=none,
    			]
 				\addplot [domain =-1:5 ,smooth, samples = 244,blue ]{sqrt(10*(1-cos(180*x/3.14))- 2*(1-cos(180*2*x/3.14)))};
 				\addplot [domain =-1:5 ,smooth, samples = 244,blue ]{-sqrt(10*(1-cos(180*x/3.14))- 2*(1-cos(180*2*x/3.14)))};
  				\addplot [domain =-1:5 ,smooth, samples = 244,red ]{1.5*x};

 				\node[black] at (axis cs:4.7,-0.5) {$k$};
 				\node[black] at (axis cs:3.14,-0.05) {$|$};
  				\node[black] at (axis cs:3.14,-0.5) {$\pi$};

  				\node[black] at (axis cs:0.4,4.5) {$\omega$};
  			\end{axis}
		\end{tikzpicture}
		\newline
		\newline
  		\begin{tikzpicture}
  			\begin{axis}
  				[
  					width = 1.2\textwidth,
  					height = 1.2\textwidth,
  					xmin = -3, xmax = 3,
  					ymin = -3, ymax = 3,
  					axis lines = center,
  					ticks = none,
  					xlabel = {\small$\operatorname{Re}$},
  					ylabel = {\small$\operatorname{Im}$},
  					x label style = {anchor = north east}
  				]
  				\addplot[only marks, red] table{
  					0 2
  					0 1
  					0 -1
  					0 -2
  				};
  			\end{axis}
  		\end{tikzpicture}
    \end{minipage}
      \caption{Intersection points of the line $ k \mapsto ck $  and
the curves $ k \mapsto \pm\omega(k) $ correspond to central
eigenvalues of the linearized spatial dynamics formulation.
Left upper panel: For $ c > c^* \approx 1.66 $, except for the trivial
solution $ k=0 $, no other intersection points occur
and so no off-zero neutral eigenvalues are present in the spatial dynamics formulation,
cf. left lower panel.
Middle upper panel: For $ c=c^* $, i.e. the maximum of $k \mapsto \omega(k)/k$, a tangent intersection occurs
leading to two double off-zero neutral eigenvalues, cf. middle  lower panel. Right upper panel: For $ c< c^*$
four non-trivial intersections occur leading to four off-zero central
eigenvalues, cf. right lower panel.}
    \label{fig3}
\end{figure}

\medskip

{\bf Acknowledgement.} Funded by the Deutsche Forschungsgemeinschaft (DFG, German Research Foundation) -- Project-ID 258734477 -- SFB 1173.

\section{The spatial dynamics formulation}
\label{sec2}

Following the approach outlined in~\S\ref{sec:approach}, we study the spatial dynamics formulation~\eqref{spatdyn}. As in~\cite{IoossJames05} we introduce $z(\xi) = v(\xi)$, $y(\xi) = v'(\xi)$ and $ U\left( \xi,p\right) =v\left( \xi+p\right) $ with $ \xi \in \R $ and $ p\in[-2,2] $. This allows us to rewrite the advance-delay differential equation~\eqref{spatdyn} as a classical dynamical system
\begin{equation}
\partial_\xi V = L_c V  + \frac{1}{c^2}N(V),
\label{eq:spatdyn}
\end{equation}
posed on the Banach space
\begin{align*}
\mathcal{H} = \{(z,y,U) \in \R \times \R \times C^0([-2,2]) : z = U(0)\},
\end{align*}
equipped with the norm $\|(z,y,U)\| = |z| + |y| + \|U\|_\infty$, where the linear operator $L_c$ acting on $\mathcal H$ with domain
\begin{align*}
	\mathcal{D} = \{(z,y,U) \in \R \times \R \times C^1([-2,2]) : z = U(0)\},
\end{align*}
is given by
\[
L_c \begin{pmatrix} z \\ y \\ U\end{pmatrix}  = \begin{pmatrix}
y  \\
\dfrac {1}{c^{2}}(5\left( U\left(1\right) -2 U\left(0\right) +U\left(-1\right) \right)
- (U\left(2\right) -2 U\left(0\right) + U\left(-2\right) )) \\
 \partial_p U
 \end{pmatrix},
\]
and where the nonlinearity $N \colon \mathcal D \to \mathcal D$ is given by
\begin{align*}
N\begin{pmatrix} z \\ y \\ U\end{pmatrix} &= \begin{pmatrix} 0 \\ s_1(U) \\ 0\end{pmatrix},
\end{align*}
with
\begin{align*}
s_1(U) &= N_{1}\!\left(U(1)-U(0)\right) - N_{1}\!\left(U(0) - U(-1) \right) + N_{2}\!\left(U(2) -U(0)\right) - N_{2}\!\left(U(0) -U(-2)\right),
\end{align*}
and
\[
N_1(U) = \mathcal{W}'_1(U)-5 U, \qquad N_2(U) = \mathcal{W}'_2(U) + U.
\]
The advance-delay equation~\eqref{spatdyn} admits a reversible symmetry and conserves a first integral. We show that these two fundamental properties are inherited by system~\eqref{eq:spatdyn}.

\paragraph{Reversibility.}
The dynamical system~\eqref{eq:spatdyn} is reversible under the symmetry
\begin{align*}
R \colon \mathcal H \to \mathcal H, \qquad R\begin{pmatrix} z \\ y \\ U\end{pmatrix} = \begin{pmatrix} -z \\ y \\ p \mapsto -U(-p)\end{pmatrix},
\end{align*}
since the spatial dynamics formulation~\eqref{spatdyn} is invariant under the transformation $v(\xi) \mapsto -v(-\xi) $. Indeed, if $V \colon \R \to \mathcal D$ is a solution to~\eqref{eq:spatdyn}, then so is $\widetilde{V} \colon \R \to \mathcal D$ given by $\widetilde{V}(\xi) = RV(-\xi)$, since it holds $RL_c = -L_c R$ and $RN(V) = -N(RV)$ for $V \in \mathcal D$.

\paragraph{Conservation of first integral.} One readily observes that the spatial dynamics formulation~\eqref{spatdyn} conserves a first integral, which is given by
\begin{align*}
c^2 v' - \int_0^1 \mathcal{W}_1'(v(\xi) - v(\xi-1)) d\xi - \int_0^2 \mathcal{W}_2'(v(\xi) - v(\xi-2)) d\xi.
\end{align*}
Thus, upon defining the nonlinear functional $I_1 \colon \mathcal H \times (1,\infty) \to \R$ by
\begin{align*}
I_1(z,y,U;c) &= \frac{1}{c^2 - 1}\left(c^2 y - \int_0^1 \mathcal{W}_1'(U(p) - U(p-1)) dp - \int_0^2 \mathcal{W}_2'(U(p) - U(p-2)) dp\right),
\end{align*}
we find that, if $V \colon \R \to \mathcal D$ solves~\eqref{eq:spatdyn}, then it holds $$\partial_\xi I_1(V(\xi);c) = 0, \qquad \xi \in \R.$$

\subsection{Spectral analysis}\label{sec:specAna}

We analyze the linear part $L_c$ of system~\eqref{eq:spatdyn} and its spectrum. The associated eigenvalue problem $(\lambda - L_c)V = 0$ with $V \in \mathcal{D}$ reads
\begin{align}\label{sds3}
\begin{split}
\lambda   {z} & =     {y}  ,\\
 \lambda   {y}  & =  \dfrac {1}{c^{2}}(5\left(  {U}(1) -2 {U}(0) +{U}(-1) \right) - ({U}(2)-2 {U}(0)+ {U}(-2) )), \\
 \lambda  {U}(p)
&= \partial_p  {U}(p).
\end{split}
\end{align}
Noting that the third equation in~\eqref{sds3} is solved by $ {U}(p) = e^{\lambda p} {z}$, the eigenvalue problem reduces to a linear homogeneous system in $\R^2$ with associated determinantal function $\sigfunk \colon \C \times \R \to \C$ given by
\begin{align}
		\sigfunk(\lambda;c) = c^2 \lambda^2 - 10(\cosh(\lambda) - 1) + 2(\cosh(2\lambda) - 1).
		\label{eq:linDispersion}
\end{align}
Thus, eigenvalues are located by the linear dispersion relation $\sigfunk(\lambda;c) = 0$. More precisely, $\lambda \in \C$ is an eigenvalue of $L_c$ if and only if $\sigfunk(\lambda;c) = 0$, and the algebraic multiplicity of $\lambda$ coincides with the multiplicity of $\lambda$ as a root of $\sigfunk(\cdot;c)$. On the other hand, if $\lambda \in \C$ is not a root of $\sigfunk(\cdot;c)$, then it follows from the upcoming Lemma~\ref{lem:resolvent} that the resolvent problem $(\lambda - L_c)V = F$ can be explicitly solved, and thus $\lambda$ lies in the resolvent set $\rho(L_c)$. An explicit expression for the resolvent can be found by first solving the third equation of $(\lambda - L_c)V = F$, which reduces the problem to a linear inhomogeneous system in $\R^2$. One readily obtains the following result.

\begin{lemma}\label{lem:resolvent}
Let $\lambda \in \C$ and $c \in \R$ be such that $\sigfunk(\lambda;c) \neq 0$. Then, $\lambda$ lies in the resolvent set of $L_c$, and it holds
\begin{align*}
(\lambda - L_c)^{-1} \begin{pmatrix} z \\ y \\ U\end{pmatrix} &= \frac{1}{\sigfunk(\lambda;c)} \begin{pmatrix} c^2(\lambda z + y) - \psi(U,\lambda) \\ c^2\lambda(\lambda z + y) - \lambda \psi(U,\lambda)\\ p \mapsto e^{\lambda p}\left(c^2(\lambda z+ y) - \psi(U,\lambda)\right)\end{pmatrix} - \begin{pmatrix} 0 \\ z \\ p \mapsto \displaystyle \int_0^p e^{\lambda (p-\tilde{p})}U(\tilde{p}) d \tilde{p}\end{pmatrix},
\end{align*}
with
\begin{align*}
\psi(U,\lambda) &= 5\left(\int_0^1 e^{\lambda (1-p)}U(p) dp - \int_{-1}^0 e^{\lambda (-1-p)}U(p) dp\right)\\
&\qquad - \left(\int_0^2 e^{\lambda (2-p)}U(p) dp - \int_{-2}^0 e^{\lambda (-2-p)}U(p) dp\right).
\end{align*}
\end{lemma}
We set $ \lambda = \lambda_r  + i \lambda_i $ with $\lambda_r,\lambda_i \in \R$. Taking real and imaginary parts of the linear dispersion relation $\sigfunk(\lambda;c) = 0$ yields the equations
\begin{align} \label{e:sysdisp}
\begin{split}
c^2 (\lambda_r^2-\lambda_i^2) & =  10\left(\cosh(\lambda_r) \cos(\lambda_i)-1\right) - 2\left(\cosh(2\lambda_r) \cos(2 \lambda_i) - 1\right) ,\\
c^2 \lambda_r\lambda_i & =  5 \sinh(\lambda_r) \sin(\lambda_i) - \sinh(2 \lambda_r) \sin(2 \lambda_i).
\end{split}
\end{align}
The equations~\eqref{e:sysdisp} can be solved graphically by plotting the solution sets of the
first and second equation in the $ (\lambda_r,\lambda_i) $-plane.
The intersection points correspond to the eigenvalues of the operator $L_c$.
\begin{figure}
    \centering
	\includegraphics[width=\textwidth]{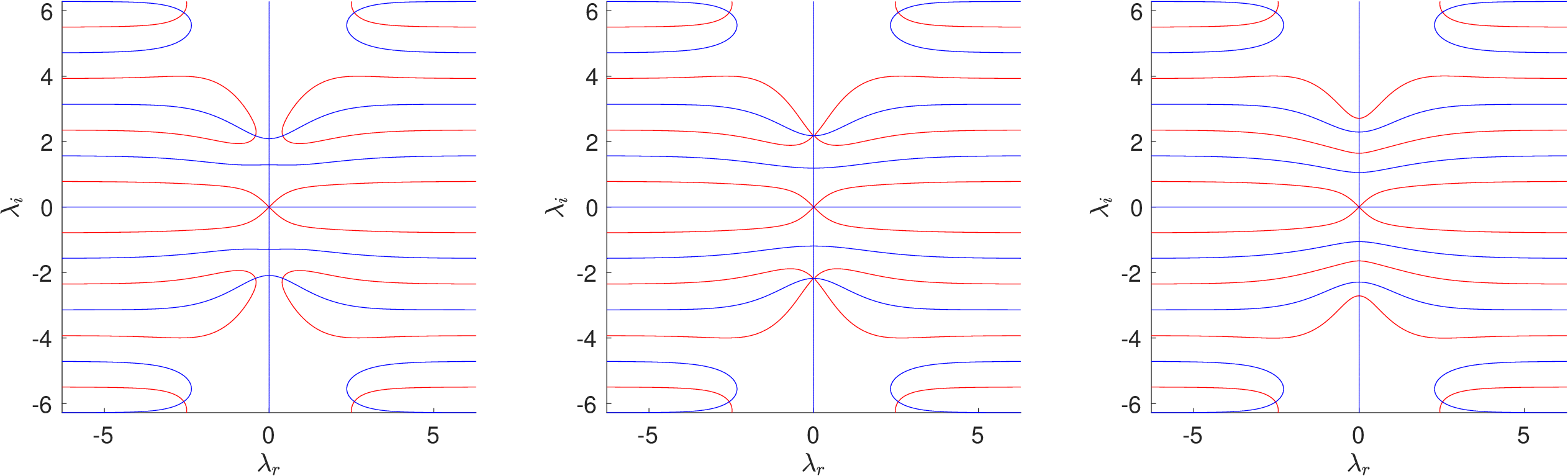}
   	\caption{The intersection points of the blue and red curves correspond to eigenvalues
of the spatial dynamics formulation, in the left panel for $ c^2 = 2.9 $,
in the middle  panel for $ c^2 = 2.743 $,  and in the right panel for $ c^2 = 2.5 $.}
    \label{fig4}
\end{figure}	
Figure~\ref{fig4} indicates that $0$ is an eigenvalue of algebraic multiplicity $2$ for all $c > 1$. We show that this is indeed the case and compute the associated spectral projection.

\begin{lemma}\label{lem:centralSpec0}
For each $c > 1$ the linear operator $L_c$ has an eigenvalue of geometric multiplicity 1 and algebraic multiplicity 2 at the origin. Associated (generalized) eigenvectors are
\begin{align*} V_0 = \begin{pmatrix} 1 \\ 0 \\ 1\end{pmatrix}, \qquad V_1 = \begin{pmatrix} 0 \\ 1 \\ p \mapsto p\end{pmatrix},
\end{align*}
with
\begin{align*}
L_c V_0 = 0, \qquad L_c V_1 = V_0.
\end{align*}
The spectral projection $\pi_{0,c} \colon \mathcal H \to \mathcal H$ onto the generalized eigenspace of $L_c$ at $0$ is explicitly given by
\begin{align}
\pi_{0,c}(V) = \chi_{0,c}^*(V) V_0 + \chi_{1,c}^*(V) V_1, \label{eq:expPi0}
\end{align}
where the functionals $\chi_{0,c}^*, \chi_{1,c}^* \colon \mathcal H \to \R$ are given by
\begin{align*}
\chi_{0,c}^*(V) &= \frac{1}{c^2 - 1}\left(c^2 z - 5\int_{-1}^1 (1-|p|)U(p) dp + \int_{-2}^2 (2-|p|)U(p) dp\right),\\
\chi_{1,c}^*(V) &= \frac{1}{c^2 - 1}\left(c^2 y - 5\int_{-1}^1 \mathrm{sgn}(p)U(p) dp + \int_{-2}^2 \mathrm{sgn}(p)U(p) dp\right),
\end{align*}
and satisfy $\chi_{i,c}^*(V_j) = \delta_{ij}$ for $i,j \in \{0,1\}$.
\end{lemma}
\begin{proof}
It holds $\sigfunk(0;c) = \partial_\lambda \sigfunk(0;c) = 0$ and $\partial_{\lambda}^2 \sigfunk(0;c) = 2(c^2 - 1) > 0$. So, $0$ is an eigenvalue of $L_c$ of algebraic multiplicity 2. By direct verification one shows $L_c V_0 = 0$ and $L_c V_1 = V_0$, which proves $0$ has geometric multiplicity 1. Finally, the associated spectral projection $\pi_{0,c}$ is given by the Dunford integral
\begin{align*}
\pi_{0,c} = \dfrac{1}{2\pi i} \int_{\Gamma_{0,c}} (\lambda - L_c)^{-1} \,d\lambda,
\end{align*}
where $\Gamma_{0,c}$ is a simple contour enclosing the origin, but no other eigenvalues of $L_c$. Applying the residue formula in~\cite[Lemma 4.1]{jamesSire05} and using the explicit expression for the resolvent from Lemma~\ref{lem:resolvent} yields~\eqref{eq:expPi0}.
\end{proof}

We now show that the neutral spectrum of $L_c$ is indeed of the form depicted in Figures~\ref{fig3} and~\ref{fig4}.
	
\begin{lemma}\label{lem:centralSpec}
There exist a unique $c^\ast > 1$ and $k_0 > 0$ such that $\sigma(L_{c^\ast}) \cap i\R = \{0,\pm i k_0\}$, where $0,\pm i k_0$ are eigenvalues of $L_{c^\ast}$ of geometric multiplicity $1$ and algebraic multiplicity $2$. Associated (generalized) eigenvectors are
\begin{align*}
V_2 = \begin{pmatrix} 1 \\ ik_0 \\ p \mapsto e^{ik_0p}\end{pmatrix}, \quad V_3 = \begin{pmatrix} 0 \\ 1 \\ p \mapsto pe^{ik_0p}\end{pmatrix},
\end{align*}
with
\begin{align*}
L_{c^\ast} V_2 = ik_0 V_2, \qquad L_{c^\ast} V_3 = ik_0V_3 + V_2.
\end{align*}
In addition, it holds
\begin{align*} \partial_{\lambda}^2 \sigfunk(ik_0;c^\ast) > 0, \qquad \sigfunk(2ik_0;c^\ast) < 0.\end{align*}
 \end{lemma}
 \begin{proof}
 	We start by showing that there exist a unique $c^\ast > 1$ and a $k_0 > 0$ such that $\Sigma(ik;c^\ast) \leq 0$ for each $k \in \R$ and the only roots of $\Sigma(\cdot;c^\ast)$ on $i\R$ are the double roots at $0$ and at $\pm ik_0$.
 	Therefore, we define $f_c \colon \R \to \R$ by
 	\begin{align*}
 		f_c(k) = \Sigma(ik;c) = -c^2 k^2 + 8 - 10 \cos(k) + 2\cos(2k),
 	\end{align*}
	and let $\Omega$ be the set of all $c > 1$ such that $f_c$ has a non-negative local maximum at some $\bar{k} > 0$.
	We note that for $c > 1$ it holds $f_c(0) = 0$, $\partial_k f_c(0) = 0$ and $\partial_k^2 f_c(0) = -2 (c^2-1) < 0$ and that $f_c(k)$ tends towards $-\infty$ for $k \rightarrow +\infty$.
	Additionally using that $f_c(\pi) = 20 - c^2 \pi^2 > 0$ for $c > 1$ close to one, we find that $\Omega$ is non-empty.
	Finally, the fact that $f_c'(k) < 0$ for all $k > 0$, provided $c > 1$ is sufficiently large, yields that $\Omega$ is bounded from above.
	Therefore, the supremum of $\Omega$ exists and we define $c^\ast := \sup \Omega$.
	
	A direct consequence of this definition is that $f_{c^\ast}(k) \leq 0$ and that there exists at least one double root $k_0 > 0$ of $f_{c^\ast}$ since $f_{c^\ast}$ has a local maximum at $k_0$.
	We now show that there cannot be more than one such $k_0$.
	For that we write $f_{c^\ast}(k) = -(c^\ast)^2 k^2 + \tilde{f}(k)$ and note that $\tilde{f}$ is independent of $c$ and $2\pi$-periodic.
	We start by showing that any positive double root $k_0$ of $f_{c^\ast}$ satisfies $k_0 \in (0,\pi)$.
	If $k_0 > 2\pi$ then
	\begin{align*}
 		f_{c^\ast}(k_0 - 2\pi) = -(c^\ast)^2 (k_0 - 2\pi)^2 + \tilde{f}(k_0 - 2\pi) > -(c^\ast)^2 k_0^2 + \tilde{f}(k_0) = f_{c^\ast}(k_0) = 0,
	\end{align*}
	which contradicts the definition of $c^\ast$.
	Similarly, if $k_0 \in (\pi,2\pi)$ we have $f_{c^\ast}(-(k_0-2\pi)) > 0$ using the reflection symmetry of $f_{c^\ast}$ and $\vert k_0-2\pi \vert < k_0$.
	Hence, any positive double root $k_0$ of $f_{c^\ast}$ must satisfy $0 < k_0 \leq \pi$.
	However, since $\partial_k^2 f_{c^\ast}$ can have at most two isolated roots in $(0,\pi)$ and $f_{c^\ast}$ has a local maximum at $k=0$, there can only be one local maximum in $(0,\pi)$.
	Therefore, the only roots of $f_{c^\ast}$ are the double roots at $0$ and at $\pm k_0$.
	Finally, since $c \mapsto f_{c}(k)$ is strictly monotonically decaying for any fixed $k > 0$ there cannot be another such $c^\ast$.
	
	Since $\Sigma(\lambda;c)$ characterizes the spectrum of $L_{c}$, we thus proved the first part of the lemma.
	Furthermore, one readily verifies that $L_{c^\ast} V_2 = ik_0 V_2$ and $L_{c^\ast} V_3 = ik_0 V_3 + V_2$, which also proves that $\pm ik_0$ are eigenvalues with geometric multiplicity 1.
	Finally, we have $\Sigma(2ik_0;c^\ast) < 0$ by construction of $c^\ast$.
	Additionally,
	\begin{align*}
		\partial_\lambda^2 \Sigma(ik_0;c^\ast) = -\partial_k^2 f_{c^\ast}(k_0) \geq 0,
	\end{align*}
	since $f_c^\ast$ has a local maximum at $k_0$.
	Now suppose that $\partial_k^2 f_{c^\ast}(k_0) = 0$, which also implies that $\partial_k^3 f_{c^\ast}(k_0) = 0$ due to the construction of $c^\ast$.
	However, the latter implies that $\partial_k^2 f_{c^\ast}(k) \leq 0$ for all $k \in (0,\pi)$ and hence, $\partial_k f_{c^\ast}$ is monotonically decreasing on $(0,\pi)$.
	Since $\partial_k f_{c^\ast}(0) = 0$ this implies that $\partial_k f_{c^\ast}(k) \leq 0$ for all $k \in (0,\pi)$ and therefore, $f_{c^\ast}$ is monotonically decreasing on $(0,\pi)$ and thus cannot have a local maximum on $(0,\pi)$ -- a contradiction.
	We therefore obtain $\partial_k^2 f_{c^\ast}(k_0) > 0$, which completes the proof.
 \end{proof}

\begin{corollary} \label{cor:centralSpec}
Let $c^\ast$ and $k_0$ be as in Lemma~\ref{lem:centralSpec}. There exists a neighborhood $\mathcal V\subset\R$ of $c^\ast$ such that, for each $c \in \mathcal V$, there exists eigenvalues $\lambda_\pm(c)$ of $L_c$ given by
\begin{align*} \lambda_\pm(c) = \pm \sqrt{s_0(c)} + i(k_0+p_0(c)), \end{align*}
with smooth functions $s_0,p_0 \colon \mathcal V\to\R$ satisfying $s_0(c^\ast) = 0 = p_0(c^\ast)$ and
\begin{align*} s_0'(c^\ast) = -\frac{2\partial_c \sigfunk(ik_0;c^\ast)}{\partial_{\lambda}^2 \sigfunk(ik_0;c^\ast)} > 0, \qquad p_0'(c^\ast) = \frac{i \partial_{c}\partial_\lambda \sigfunk(ik_0;c^\ast)}{\partial_{\lambda}^2 \sigfunk(ik_0;c^\ast)}.\end{align*}
\end{corollary}
\begin{proof}
The result follows by Riemann surface unfolding of the linear dispersion relation $\sigfunk(\lambda;c) = 0$ or, equivalently, of system~\eqref{e:sysdisp}. More precisely, we substitute $\lambda_r = \sqrt{\gamma}$ in~\eqref{e:sysdisp} and observe that both equations are analytic in $\gamma, \lambda_i$ and $c$, where we use that both $\cosh(\sqrt{\gamma})$ and $\sinh(\sqrt{\gamma})/\sqrt{\gamma}$ are analytic functions of $\gamma$. Subsequently, we use Lemma~\ref{lem:centralSpec} and apply the implicit function theorem at $\gamma = 0$, $\lambda_i = k_0$ and $c = c^\ast$ to solve the obtained system for $\gamma$ and $\lambda_i$ as smooth functions of $c$. Expanding $\gamma$ and $\lambda_i$ in $c$ yields the result.
\end{proof}

	\begin{remark}\label{rem:centralSpec}
	We highlight that Lemma~\ref{lem:centralSpec} and Corollary~\ref{cor:centralSpec} show that for $c \in \mathcal V$ with $c < c^\ast$ the double eigenvalues at $\pm ik_0$ split up into two simple purely imaginary eigenvalues, whereas for $c > c^\ast$ these double eigenvalues split up into two simple eigenvalues with opposite real parts lying on the lines $\mathrm{Im}(\lambda) = \pm k_0$, as depicted in Figure~\ref{fig3}
	\end{remark}
	
	Since $\sigfunk(\cdot;c^\ast)$ is analytic, its roots are isolated. Thus, we can proceed as in~\cite[Lemma 1(i)]{KiIo00} to bound the hyperbolic part of the spectrum of $L_c^\ast$
	
	\begin{lemma}\label{lem:spectralGap}
Let $c^\ast$ be as in Lemma~\ref{lem:centralSpec}. There exists $\rho > 0$ such that any eigenvalue $\lambda \in \sigma(L_{c^\ast}) \setminus i\R$ satisfies $\vert\operatorname{Re}(\lambda)\vert > \rho$.
	\end{lemma}
	
	The above Lemmas~\ref{lem:centralSpec} and~\ref{lem:spectralGap} suggest that we are in a situation where the dynamics of small solutions to the full spatial system~\eqref{eq:spatdyn} can potentially be reduced using center manifold theory, cf.~\cite{HaIo11}.
	Indeed, in the following~\S\ref{sec3} we rigorously apply a center manifold reduction to~\eqref{eq:spatdyn} for $c$ close to $c^\ast$.
	Since we want to exploit the spectral properties of $L_c$ at $c = c^\ast$, see Lemma~\ref{lem:centralSpec}, we rewrite~\eqref{eq:spatdyn} as
	\begin{align}
		\partial_\xi V = L_{c^\ast} V + \left(L_c - L_{c^\ast}\right) V + \dfrac{1}{c^2}N(V) =: L_{c^\ast} V + \curlN(V;c).
		\label{eq:spatDynCast}
	\end{align}
	We remark that the range of $\curlN(\cdot;c)$ is still spanned by the vector $(0,1,0)^\top$, which we will exploit to obtain a center manifold theorem.
	Additionally, we note that $\curlN(\cdot;c)$ has a linear part in $V$ which vanishes at $c = c^\ast$.
	In particular, it satisfies $\curlN(0;c^\ast) = 0$ and $\partial_V \curlN(0;c^\ast) = 0$, which suffices for our purposes, cf.~\cite[Hypothesis 2.3.1]{HaIo11}.

\subsection{Removing the translational mode} \label{sec:trans}

Since the spatial dynamics formulation~\eqref{spatdyn} is translationally invariant, system~\eqref{eq:spatdyn}, and thus system~\eqref{eq:spatDynCast}, is invariant under the shift $V \mapsto V + \tau V_0$ for any $\tau \in \R$. As outlined in~\S\ref{sec:approach}, our approach is to first factor out the associated eigenvalue at the origin of $L_{c^\ast}$ before applying the center manifold reduction, so that we can work with a five- instead of a six-dimensional center manifold.

We decompose any solution $V \colon \R \to \mathcal D$ to~\eqref{eq:spatDynCast} as
\begin{align} \label{def:q} V(\xi) = W(\xi) + \tau(\xi)V_0, \qquad \tau(\xi) = \chi_{0,c^\ast}^*(V(\xi)), \end{align}
so that $\chi_{0,c^\ast}^*(W(\xi)) = 0$ for all $\xi \in \R$. Applying the spectral projection $\pi_{0,c^\ast}$ of $L_{c^\ast}$ and using $\chi_{0,c^\ast}^*(W) = 0$, $\chi_{0,c^\ast}^*(\curlN(V;c)) = 0$ and $\curlN(W+\tau V_0;c) = \curlN(W;c)$, we find that system~\eqref{eq:spatDynCast} in the new $W$- and $\tau$-variables reads
\begin{align}
\partial_\xi W &= \tildeL W + \widetilde{\curlN}(W;c), \label{e:SD1}\\
\partial_\xi \tau &= \chi_{1,c^\ast}^*(W), \label{e:SD2}
\end{align}
where the linear operator $\tildeL$ acting on
\begin{align*} \tildeH = \{W \in \mathcal H : \chi_{0,c^\ast}^*(W) = 0\} = \ker(\chi_{0,c^\ast}^*),\end{align*}
with domain
\begin{align*}
	\tildeD = \tildeH \cap \mathcal D,
\end{align*}
is given by
\[
\tildeL(W) = L_{c^\ast}(W) - \chi_{1,c^\ast}^*(W)V_0,
\]
and where the nonlinearity $\widetilde{\curlN}$ is the restriction of $\curlN$ to $\tildeD$. We emphasize that the $W$-equation in~\eqref{e:SD1}--\eqref{e:SD2} is decoupled from the $\tau$-equation, which captures the shift action of~\eqref{eq:spatDynCast}. Thus, we have factored out the translational mode, and $0$ is now a simple eigenvalue of $\tildeL$ with eigenvector $V_1$. Because the functional $\chi_{0,c^\ast}^*$ must vanish on all (generalized) eigenvectors except $V_0$, cf.~Lemma~\ref{lem:centralSpec0}, one readily observes that the spectrum of $\tildeL$, outside the origin, is the same as the spectrum of $L_{c^\ast}$ with the same multiplicities of eigenvalues and the same (generalized) eigenvectors.

\section{The center manifold reduction} 

\label{sec3}

We now apply center manifold theory to the spatial system~\eqref{e:SD1}, which is precisely of the form as discussed in~\cite[Section 3.2.1]{HaIo11}.
This application requires that the spectrum of $\tildeL$ can be split into a neutral part $\sigma_n = \{0,\pm ik_0\}$ and a remaining hyperbolic part $\sigma_h$, which has been established in the previous~\S\ref{sec2}.
Additionally, we need to establish an optimal regularity estimate for the affine linearized problem on the hyperbolic eigenspace, namely
\begin{align}
	\partial_\xi W_h = \tildeLh W_h + F_h,
	\label{eq:affineProblem}
\end{align}
where $\tildeLh = \tildeL \tildepih$ and $F_h = \tildepih F$ with $\tildepih$ being the spectral projection onto the hyperbolic eigenspace of $\tilde{L}$.
Here, we can restrict to $F = (0,f,0)^\top$ with $f \colon \R \to \R$ due to the special form of the nonlinearity in~\eqref{eq:spatDynCast}.
We proceed as in~\cite{KiIo00,IoossJames05,jamesSire05} and solve~\eqref{eq:affineProblem} explicitly.
More precisely, we show that there exists $\alpha_0 > 0$ such that for all $\alpha \in [0,\alpha_0)$ the affine problem~\eqref{eq:affineProblem} has a unique solution $W_{h,\alpha}$ in $E^\alpha_0(\tildeDh) \cap E^\alpha_1(\tildeHh)$ for each $f \in E_0^\alpha(\R)$.
Here, $\tildeDh = \tildepih \tildeD$ and $\tildeHh = \tildepih \tildeH$, and the space $E^\alpha_j(Z)$ is defined by
\begin{align*}
	E^\alpha_j(Z) = \left\{ f \in C^j(\R,Z) \,:\, \| f \|_j = \max_{0 \leq k \leq j} \sup_{\xi \in \R} e^{-\alpha\vert \xi \vert} \| D^k f\| < \infty\right\},
\end{align*}
for an arbitrary Banach space $Z$, $j \in \N$ and $\alpha \geq 0$.
Additionally, we show that the linear solution operator $K_\alpha : E^\alpha_0(\R) \rightarrow E^\alpha_0(\tildeDh)$ mapping $f \mapsto W_{h,\alpha}$ is $\alpha$-uniformly bounded.

\begin{remark}
	We point out that, as discussed in~\cite{jamesSire05}, the operator $\tildeL$ is not bi-sectorial.
	Therefore, the resolvent estimates in~\cite[Hypothesis 2.2.15]{HaIo11} are not available to establish a center manifold result.
	Instead, we solve the affine problem~\eqref{eq:affineProblem} directly in order to obtain the desired optimal regularity condition for the center manifold result, see also~\cite[Section 5.2.3]{HaIo11}.
\end{remark}

The remainder of this section is organized as follows. We first construct the spectral projections $\tildepin$ and $\tildepih = 1-\tildepin$ by explicitly computing the resolvent $(\lambda - \tildeLh)^{-1}$.
Then, we solve~\eqref{eq:affineProblem} for $\alpha = 0$ and as a final step, we extend this result to hold for $\alpha > 0$ sufficiently small.
Finally, we state the center manifold result.

\subsection{Spectral projections}

The neutral spectral projection $\tildepin$ of $\tildeL$ is given by the Dunford integral
\begin{align}
	\tildepin = \dfrac{1}{2\pi i} \int_{\Gamma_n} (\lambda - \tildeL)^{-1} \,d\lambda,
	\label{eq:neutralProjection}
\end{align}
where $\Gamma_n \subset \C$ is a regular curve enclosing the neutral eigenvalues in $\sigma_n = \{0,\pm ik_0\}$ but excluding the hyperbolic ones in $\sigma_h = \sigma(\tildeL) \setminus \sigma_n$, cf.~Lemmas~\ref{lem:centralSpec} and~\ref{lem:spectralGap}.
We then explicitly compute
\begin{align}
	(\lambda - \tildeL)^{-1} \begin{pmatrix}
		0 \\ f \\ 0
	\end{pmatrix} = (c^\ast)^2 f \left[\dfrac{1}{\Sigma(\lambda;c^\ast)} \begin{pmatrix}
		1 \\ \lambda \\ p \mapsto e^{\lambda p}
	\end{pmatrix} - \dfrac{1}{\lambda^2 ((c^\ast)^2 - 1)} V_0\right],
	\label{eq:resolventLTilde}
\end{align}
for $\lambda \notin \sigma(\tildeL)$ and $f \in \R$.
To check this we recall that $\tildeL (W) = L_{c^\ast} (W) - \chi_{1,c^\ast}^\ast(W) V_0$.
Then, using Lemma~\ref{lem:resolvent} we find that
\begin{align*}
	(\lambda - L_{c^\ast})^{-1} \begin{pmatrix}
		0 \\ f \\ 0
	\end{pmatrix} = \dfrac{(c^\ast)^2 f}{\Sigma(\lambda;c^\ast)} \begin{pmatrix}
		1 \\ \lambda \\ p \mapsto e^{\lambda p}
	\end{pmatrix} =: v,
\end{align*}
and
\begin{align*}
	\chi_{1,c^\ast}^\ast(v) = \dfrac{(c^\ast)^2 f}{\lambda ((c^\ast)^2 - 1)}.
\end{align*}
With this, $L_{c^\ast} V_0 = 0$ and $\chi_{1,c^\ast}(V_0) = 0$ one can directly verify that~\eqref{eq:resolventLTilde} holds true.
Using the explicit expression for the resolvent~\eqref{eq:resolventLTilde}, we obtain the following result.

\begin{proposition}\label{prop:hyperbolicProjection}
	Let $F \in \tildeH$ be a vector of the form $F = (0,f,0)^T$ with $f \in \R$. Then, the projection of $F$ onto the hyperbolic eigenspace of $\tildeL$ is given by
	\begin{align*}
		F_h = \begin{pmatrix}
			0 \\ k_1 c^2 f \\ p \mapsto c^2 f k_2(p)
		\end{pmatrix},
	\end{align*}
	with $k_1 \in \R$ and $k_2 \in C^\infty([-2,2])$ satisfying $k_2(p) = k_2(-p)$.
\end{proposition}
\begin{proof}
	We note that, since $\lambda \mapsto (\lambda - \tildeL)^{-1}$ is analytic on the resolvent set of $\tildeL$, the projection $\tildepin$ onto the neutral eigenspace of $\tildeL$ is given by the sum of its residues at $\lambda = \pm ik_0$ and $\lambda =0$.
	Furthermore, since $\Sigma(\cdot;c^\ast)$ is an analytic function with double roots in $\sigma_n = \{0,\pm ik_0\}$, see~\S\ref{sec2}, these residues can be calculated explicitly using~\cite[Lemma 4.1]{jamesSire05}.
	The result then follows from this explicit formula, the symmetry $\Sigma(\lambda;c^\ast) = \Sigma(-\lambda;c^\ast)$ and the fact that $\sigma_n$ is purely imaginary and invariant under the reflection $\lambda \mapsto -\lambda$.
\end{proof}

	\subsection{The affine problem for \texorpdfstring{$\alpha = 0$}{ }}
	We now solve the affine problem~\eqref{eq:affineProblem} in $E^0_0(\tildeDh) \cap E_1^0(\tildeHh) = C^0_b(\R,\tildeDh) \cap C_b^1(\R,\tildeHh)$, where $W_h = (z_h, y_h, U_h)$ and $F_h = (0,k_1 c^2 f, p \mapsto k_2(p) c^2 f)^T$ with $f \in C_b^0(\R)$, cf.~Proposition~\ref{prop:hyperbolicProjection}.
	Using that the third equation of~\eqref{eq:affineProblem} is an inhomogeneous transport equation we obtain via the Duhamel formula that
	\begin{align*}
		U_h(\xi,p) = z_h(\xi + p) + \int_\xi^{\xi + p} k_2(\xi+p-s) c^2 f(s) \,ds,
	\end{align*}
	with $p \in [-2,2]$ and $\xi \in \R$, where we used $U_h(\xi,0) = z_h(\xi)$.
	Therefore, we have the estimate
	\begin{align}
		\|U_h\|_{C^0_b(\R,C^1([-2,2]))} \leq \|z_h\|_{C^1} + C \|c^2 f\|_{C^0},
		\label{eq:Uestimate}
	\end{align}
	for some constant $C > 0$.
	It therefore remains to provide an estimate for $z_h$.
	Since $\alpha = 0$, and thus the problem is posed in $C^0_b(\R,\tildeDh)$ we can take the Fourier transform of~\eqref{eq:affineProblem} with respect to $\xi$ in the space of tempered distributions and find
	\begin{align}
		(ik - \tildeLh)\hat{W}_h = \hat{F}_h,
		\label{eq:affineProbFT}
	\end{align}
	with $k \in \R$.
	We note that $ik-\tildeLh$ is invertible on $\tildeHh$ with an analytic inverse in a strip around the real axis using Lemma~\ref{lem:spectralGap} and recalling that we projected the equation onto the hyperbolic eigenspace.
	Utilizing $\chi_{1,c^\ast}(\tildepih W) = 0$ for all $W \in \tildeH$ we find
	\begin{align*}
		(ik - \tildeLh) \begin{pmatrix}
			\hat{z}_h \\ \hat{y}_h \\ \hat{U}_h
		\end{pmatrix}\! &= \!\begin{pmatrix}
			ik \hat{z}_h - \hat{y}_h \\
			ik \hat{y}_h - \frac{1}{c^2}\left[5\left(\hat{U}_h(1) - 2 \hat{U}_h(0) + \hat{U}_h(-1)\right) - \left(\hat{U}_h(2) - 2\hat{U}_h(0) - \hat{U}_h(-2)\right)\right] \\
			ik \hat{U}_h - \partial_p \hat{U}_h
		\end{pmatrix}.
	\end{align*}
	Thus, integrating the third equation in~\eqref{eq:affineProbFT} yields
	\begin{align*}
		\hat{U}_h(k,p) = e^{ikp} \hat{z}_h(k) - c^2 \hat{f}(k) \int_0^pe^{ik(p-s)} k_2(p) \,ds
	\end{align*}
	where we used that $\hat{z}_h(k) = \hat{U}_h(k,0)$.
	Inserting this into the second equation of~\eqref{eq:affineProbFT}, we then obtain
	\begin{align*}
		-k^2 \hat{z}_h(k) &- \dfrac{1}{c^2}\bigg[5\bigg(e^{ik}\hat{z}_h(k) - c^2 \hat{f}(k) \int_0^1 e^{ik(1-s)} k_2(s) \,ds \\
		&\qquad\qquad - 2 \hat{z}_h(k) + e^{-ik}\hat{z}_h(k) - c^2 \hat{f}(k) \int_0^{-1} e^{ik(-1-s)} k_2(s) \,ds\bigg) \\
		&\qquad -\bigg(e^{2ik}\hat{z}_h(k) - c^2 \hat{f}(k) \int_0^2 e^{-ik(2-s)} k_2(s) \,ds \\
		&\qquad\qquad - 2 \hat{z}_h(k) + e^{-2ik}\hat{z}_h(k) - c^2 \hat{f}(k) \int_0^{-2} e^{ik(-2-s)} k_2(s) \,ds\bigg)\bigg] \\
		&= k_1 c^2 \hat{f}(k).
	\end{align*}
	This is equivalent to
	\begin{align}
		\sigfunk(ik;c) \hat{z}_h(k) = c^2 \hat{f}(k) \hat{h}(k),
		\label{eq:affineProbFT2}
	\end{align}
	with
	\begin{align*}
		\hat{h}(k) = k_1 -10\int_0^1 k_2(s) \cos(k(1-s)) \,ds + 2 \int_0^2 k_2(s) \cos(k(2-s)) \,ds.
	\end{align*}
	Here, we used $k_2(s) = k_2(-s)$.
	Finally, we rewrite~\eqref{eq:affineProbFT2} as
	\begin{align}
		\sigfunk(ik;c) \left[\hat{z}_h(k) - c^2 \hat{f}(k) \hat{H}(k)\right] = 0.
		\label{eq:affineProbFTEquivalent}
	\end{align}
	Note that, since~\eqref{eq:affineProbFTEquivalent} is equivalent to~\eqref{eq:affineProbFT} and $(ik-\tildeLh)^{-1}$ is analytic for $k$ in a strip around the real axis, so is $\hat{H}$.
	Recall that the above equation is posed in the space of tempered distributions and thus, by using~\cite[Appendix A]{KiIo00} and the fact that $\sigfunk(ik;c)$ has a double root at $k \in \{0,\pm k_0\}$ we obtain
	\begin{align*}
		\hat{z}_h - c^2 \hat{f} \hat{H} = \alpha_+ \delta_{k_0} + \beta_+ \delta_{k_0}^\prime + \alpha_0 \delta_0 + \beta_0 \delta_0^\prime + \alpha_- \delta_{-k_0} + \beta_- \delta_{-k_0}^\prime,
	\end{align*}
	for some $\alpha_\pm, \beta_\pm, \alpha_0, \beta_0 \in \R$, where $\delta_k$ is the Dirac distribution at $k$ and $\delta'_k$ its distributional derivative.
	Finally, since $\sigfunk(ik;c) = \mathcal{O}(k^2)$ for $\vert k \vert \rightarrow \infty$, we find $\hat{H}(k) = \mathcal{O}(1/k^2)$ as $k \to \pm \infty$.
	Then, by analyticity of $\hat{H}$ in a strip around the real axis, Paley-Wiener theory yields that there exists a Fourier inverse $H$ of $\hat{H}$, which satisfies $H \in H^1_\varepsilon(\R)$, i.e.~$e^{\varepsilon \vert \xi \vert} H \in H^1(\R)$ for $\varepsilon > 0$ sufficiently small.
	Therefore, we can solve~\eqref{eq:affineProblem} explicitly by
	\begin{align*}
		z_h(\xi) &= c^2(H \ast f)(\xi) + (\alpha_+ + i\xi \beta_+) e^{ik_0\xi} + (\alpha_0 + i\xi \beta_0) + (\alpha_- + i\xi \beta_-) e^{-ik_0\xi}, \\
		y_h(\xi) &= z_h^\prime(\xi), \\
		U_h(\xi,p) &= z_h(\xi + p) + \int_\xi^{\xi + p} k_2(\xi + p - s) c^2 f(s) \,ds.
	\end{align*}
	Now, if $f \in E_0^\alpha(\R)$ for $\alpha < 0$ sufficiently small, then $(c^2(H \ast f), y_h,U_h)^T$ solves~\eqref{eq:affineProblem} and thus, in particular $\tildepin(c^2(H \ast f), y_h,U_h)^T = 0$.
	Since the projection is independent of $\alpha$ and acts pointwise in time, we also have that $\tildepin(c^2(H \ast f), y_h,U_h)^T = 0$ for $f \in E^0_0(\R)$ and thus, necessarily
	\begin{align*}
		\tildepin (\tilde{z}_h(\xi), \tilde{z}_h^\prime(\xi),\tilde{z}_h(\xi + p))^T = 0,
	\end{align*}
	for $\tilde{z}_h = (\alpha_+ + i\xi \beta_+) e^{ik_0\xi} + (\alpha_0 + i\xi \beta_0) + (\alpha_- + i\xi \beta_-) e^{-ik_0\xi}$.
	Recalling that we solve~\eqref{eq:affineProblem} in $\tildeHh$, we find
	\begin{align*}
		\pi_n (\tilde{z}_h(\xi), \tilde{z}_h^\prime(\xi), \tilde{z}_h(\xi+p))^T = 0,
	\end{align*}
	where $\pi_n$ is the spectral projection onto the neutral eigenspace of $L_{c^\ast}$.
	Since $\sigfunk(\lambda;c)$ has double roots at $\lambda = \pm ik_0, 0$, $(\tilde{z}_h(\xi), \tilde{z}_h^\prime(\xi), \tilde{z}_h(\xi + p))^T$ is in the generalized neutral eigenspace of $L_{c^\ast}$ and, therefore, it holds
	\begin{align*}
		\alpha_\pm = \alpha_0 = \beta_\pm = \beta_0 = 0.
	\end{align*}
	Using the above explicit solution and~\eqref{eq:Uestimate}, we establish the estimate
	\begin{align*}
		\|V_h\|_{C^0_b(\R;\tildeDh) \cap C^1(\R;\tildeHh)} \leq C \|f\|_{C^0_b(\R)},
	\end{align*}
	and hence, we arrive at the following result.
	
	\begin{lemma}\label{lem:affineProblem}
		Let $F = (0,f,0)^T$ for $f \in C^0_b(\R)$.
		Then, the affine system~\eqref{eq:affineProblem} has a unique, bounded solution $W_h \in C^0_b(\tildeDh) \cap C^1_b(\tildeHh)$ and the linear operator $K_h : C^0_b(\R) \rightarrow C^0_b(\tildeDh)$, $f \mapsto W_h$ is bounded.
	\end{lemma}

	\subsection{The affine problem for \texorpdfstring{$\alpha > 0$}{ } small and the center manifold result}
	
	After solving the affine problem~\eqref{eq:affineProblem} in $C^0_b(\tildeDh) \cap C_b^1(\tildeHh)$ in the previous section, we now turn to the case $\alpha > 0$, that is, we want to show existence of a solution to~\eqref{eq:affineProblem} in $E_0^\alpha(\tildeDh) \cap E_1^\alpha(\tildeHh)$ for $\alpha > 0$ sufficiently small.
	For that we use the following result from \cite{jamesSire05}, which we recall for completeness.
	
	\begin{lemma}[{\cite[Lemma 4.4]{jamesSire05}}]
		Let $\mathcal{D}$, $\mathcal{H}$ be Banach spaces, where $\mathcal{D}$ is continuously embedded into $\mathcal{H}$.
		Furthermore, let $L \colon \mathcal{D} \rightarrow \mathcal{H}$ be a closed linear operator such that
		\begin{align}
			\partial_\xi W = LW + f,
			\label{eq:lemmaAlphaPositive}
		\end{align}
		has a unique solution $W = Kf \in C^0_b(\R;\mathcal{D}) \cap C^1_b(\R;\mathcal{H})$ for any fixed $f \in C^0_b(\mathcal{D})$ and $K$ is a bounded linear operator from $C^0_b(\R;\mathcal{D})$ to $C^0_b(\R;\mathcal{D})$.
		Then, there exists an $\alpha_0 > 0$ such that for all $\alpha \in [0,\alpha_0)$ and $f \in E^\alpha_0(\mathcal{D})$ the system~\eqref{eq:lemmaAlphaPositive} has a unique solution $W_\alpha$ in $E^\alpha_0(\mathcal{D}) \cap E^\alpha_1(\mathcal{H})$, which satisfies the estimate
		\begin{align*}
			\|W_\alpha\|_{E^\alpha_0(\mathcal{D})} \leq C(\alpha) \|f\|_{E_0^\alpha(\mathcal{D})}.
		\end{align*}
	\end{lemma}
	
	Applying this result to our setting we are able to extend the existence to the case $\alpha > 0$.
	This yields the following result.
	
	\begin{proposition}\label{prop:affineProblem}
		There exists an $\alpha_0 > 0$ such that for all $\alpha \in [0,\alpha_0)$ and $F = (0,f,0)^\top$ with $f \in E_0^\alpha(\R)$ the affine problem~\eqref{eq:affineProblem} admits a unique solution $W_{h,\alpha} \in E^\alpha_0(\tildeDh) \cap E^\alpha_1(\tildeHh)$ and the linear operator $K : E^\alpha_0(\R) \rightarrow E^\alpha_0(\tildeDh)$, $f \mapsto W_{h,\alpha}$ is $\alpha$-uniformly bounded.
	\end{proposition}

	The results in Lemma~\ref{lem:spectralGap}, Lemma~\ref{lem:affineProblem} and Proposition~\ref{prop:affineProblem} imply that the assumptions of the center manifold result~\cite[Theorem 3.3]{HaIo11} are satisfied.
	Thus, we obtain the following.
	
	\begin{theorem} \label{theo:CM}
		Let $m \in \mathbb{N}_{\geq 2}$ and let $c^\ast$ be as in Lemma~\ref{lem:centralSpec}. Let $\tildepin \colon \tildeH \to \tildeH$  be the spectral projection associated with the neutral eigenvalues $\{0,\pm ik_0\}$ of $\tildeL$. Let $\tildeDn = \tildepin\tildeD = \mathrm{Span}\{V_1,V_2,V_3,\overline{V_2},\overline{V_3}\}$ and $\tildeDh = (1-\tildepin)\tildeD$.

		There exist a neighborhood $\mathcal U \times \mathcal V$ of $(0,c^\ast)$ in $\tildeDn \times \R$ and a map $\psi \in C^{2m}(\mathcal U \times \mathcal V; \tildeDh)$, such that the following assertions hold true for all $c \in \mathcal V$:
		\begin{itemize}
			\item We have $\psi(0;c_*) = \partial_W \psi(0;c^*) = 0$.
            \item If $W \colon \R \to \widetilde{\mathcal{D}}_c$ solves~\eqref{e:SD1} and it holds $\tildepin(W(\xi))\in \mathcal U$ for all $\xi \in \R$, then $W_n(\xi) = \tildepin(W(\xi))$ solves
            	\begin{align} \label{e:SD3}
					\partial_\xi W_n = \tildeL W_n + \tildepin\left(\left(\widetilde{L}_{c} - \tildeL\right)\left(W_n + \psi(W_n;c)\right) + \frac{1}{c^2} \widetilde{N}(W_n + \psi(W_n;c))\right),
				\end{align}
                  and it holds $W(\xi) = W_n(\xi) + \psi(W_n(\xi);c)$ for all $\xi \in \R$.
			\item If $W_n \colon \R \to \mathcal U$ solves~\eqref{e:SD3}, then $W(\xi) = W_n(\xi) + \psi(W_n(\xi);c)$ is a solution to~\eqref{e:SD1}.
            \item The map $\psi(\cdot;c)$ commutes with $R$, and~\eqref{e:SD3} is reversible under the symmetry $R$.
		\end{itemize}
	\end{theorem}
	
	\begin{remark}\label{rem:PhiAddProp}
		We point out that since $W \equiv 0$ is a solution of~\eqref{e:SD1} for all $c > 0$, the map $\psi$ in Theorem~\ref{theo:CM} additionally satisfies $\psi(0;c) = 0$ for all $c \in \mathcal{V}$.
	\end{remark}

\section{Discussion of the reduced system}
	
\label{sec4}
	
Let $\mathcal U,\mathcal V$ be as in Theorem~\ref{theo:CM}. We study the $5$-dimensional reduced system~\eqref{e:SD3}, which governs the dynamics on the center manifold $\{W_n + \psi(W_n;c) : (W_n,c) \in \mathcal U \times \mathcal V\}$.

\subsection{Taking care of the zero mode} \label{sec41}

As outlined in~\S\ref{sec:approach}, our approach is to first reduce one dimension further by exploiting that~\eqref{spatdyn} conserves the first integral $I_1(\cdot,c^\ast)$. The linearization of this first integral is readily seen to be the functional $\chi_{1,c^\ast}^*$ arising in the spectral projection~\eqref{eq:expPi0} associated with the $0$-eigenvalue of $L_{c^\ast}$. The conservation of the first integral then yields that the dynamics on the eigenspace $\mathrm{Span}\{V_1\}$ associated with the $0$-eigenvalue of $\widetilde{L}_{c^\ast}$ is constant up to nonlinear effects.

Thus, let $c \in \mathcal V$, let $\tau_0 \in \R$ and let $W_n \colon \R \to \mathcal U$ be a solution to~\eqref{e:SD3}. Then, $W(\xi) = W_n(\xi) + \psi(W_n(\xi);c)$ is the corresponding solution to~\eqref{e:SD1} on the center manifold, whereas $V(\xi) = W(\xi) + \tau(\xi)V_0$ solves system~\eqref{eq:spatDynCast} with $\tau(\xi)$ determined by~\eqref{e:SD2} and $\tau(0) = \tau_0$. By Lemma~\ref{lem:centralSpec0} the dynamics on the eigenspace $\mathrm{Span}\{V_1\}$ is described by the function $d \colon \R \to \R$ given by
\begin{align} \label{def:d}
d(\xi) = \chi_{1,c^\ast}^*(W_n(\xi)) = \chi_{1,c^\ast}^*(W(\xi)) = \chi_{1,c^\ast}^*\left(V(\xi)\right),
\end{align}
where we used that $\psi$ maps into $\widetilde{\mathcal{D}}_h$. Setting
\begin{align} \label{def:Vn}
V_n(\xi) = W_n(\xi) - d(\xi)V_1,
\end{align}
it holds $\chi_{1,c^\ast}^*(V_n(\xi)) = 0$ for all $\xi \in \R$ and, thus, $V_n(\xi)$ lies in the complementary subspace
$$\mathcal{D}_{\neq 0} = \mathrm{Span}\{V_2,V_3,\overline{V_2},\overline{V_3}\}.$$

Since the linearization of the first integral $I_1(\cdot,c^\ast)$ is the functional $\chi_{1,c^\ast}^*$ and the solution $W_n(\xi)$ must have small amplitude to stay in the neighborhood $\mathcal U$ of $0$, it is expected that nonlinear effects are rather small and the function $d(\xi)$ is close to the constant value
$$D = I_1(W(\xi);c) = I_1(V(\xi);c),$$
for all $\xi \in \R$. This can be exploited to replace $d(\xi)$ by the parameter $D$ using a near identity change of variables, which is the content of the following lemma.

\begin{lemma} \label{lem:IFT}
Let $m \in \mathbb{N}_{\geq 2}$ and let $\mathcal U, \mathcal V,\psi$ be as in Theorem~\ref{theo:CM}. There exist neighborhoods $\mathcal Z_{1,2} \subset \R$, $\mathcal U_1 \subset \mathcal U$ and $\mathcal U_2 \subset \mathcal{D}_{\neq 0}$ of $0$ and a function $G_1 \in C^{2m}(\mathcal U_2 \times \mathcal Z_2 \times \mathcal V,\mathcal{Z}_1)$ such that the following assertions hold true:
\begin{itemize}
\item There exists a constant $C > 0$ such that
\begin{align} |G_1(V_n;D,c) - D| \leq C\left((|D| + \|V_n\|)^2 + |c-c^\ast|(|D| + \|V_n\|)\right), \label{e:IFT1}\end{align}
for all $(V_n,D,c) \in \mathcal U_2 \times \mathcal Z_2 \times \mathcal V$.
\item Let $W_n \in \mathcal U_1$ and $c \in \mathcal V$, and set $d = \chi_{1,c^\ast}^*(W_n)$, $V_n = W_n - dV_1$, $W = W_n + \psi(W_n;c)$ and $D = I_1(W;c)$. Then, we have $V_n \in \mathcal U_2$, $D \in \mathcal Z_2$ and $d = G_1(V_n;D,c)$.
\item Let $(V_n,D,c) \in \mathcal U_2 \times \mathcal Z_2 \times \mathcal V$, and set $d = G_1(V_n;D,c)$ and $W_n = V_n + dV_1$. Then, it holds $W_n \in \mathcal U$ and $D = I_1(W;c)$ with $W = W_n + \psi(W_n;c)$.
\end{itemize}
\end{lemma}
\begin{proof}
Take neighborhoods $\check{\mathcal Z} \subset \R$ and $\check{\mathcal U} \subset \mathcal{D}_{\neq 0}$ of $0$ such that $\check{\mathcal Z}V_1 + \check{\mathcal U} \subset \mathcal U$. For $d \in \check{\mathcal Z}$, $c \in \mathcal V$ and $V_n \in \check{\mathcal U}$, we set $W = W(V_n,d,c) = V_n + dV_1 + \psi(V_n + dV_1;c)$. We wish to solve the equation
\begin{align} \label{e:IFT}
\begin{split}
d - D &= \chi_{1,c^\ast}^*(W) - I_1(W;c) = \chi_{1,c^\ast}^*(W) - \chi_{1,c}^*(W)\\
&\qquad \qquad +\, \frac{1}{c^2 - 1} \left(\int_0^1 N_1(W_3(p) - W_3(p-1)) dp + \int_0^2 N_2(W_3(p) - W_3(p-2)) dp\right),
\end{split}
\end{align}
where $W_3$ denotes the third component of $W$. The right-hand side of~\eqref{e:IFT} can be bounded from above by $C( \|W\|^2 + |c-c^\ast|\|W\|)$, where $C > 0$ is some constant. Since we have $W = V_n + dV_1 + \psi(V_n + dV_1;c)$, the equation~\eqref{e:IFT} takes the abstract form $F(d,D,V_n,c) = 0$, where $F \colon \check{\mathcal Z} \times \R \times \check{\mathcal U} \times \mathcal V \to \R$ is $C^{2m}$ by Theorem~\ref{theo:CM}. Moreover, its Fr\'echet derivative is
\begin{align} F'(0,0,0,c^\ast)[d,D,V_n,c] = d - D, \label{e:IFT2}\end{align}
for $(d,D,V_n,c) \in \check{\mathcal Z} \times \R \times \check{\mathcal U} \times \mathcal V$. Hence, $F(d,D,V_n,c) = 0$ can be solved for $d$ by the implicit function theorem, which, in combination with Theorem~\ref{theo:CM}, readily yields the result.
\end{proof}

Thus, substituting $W_n(\xi) = V_n(\xi) + d(\xi)V_1$ into system~\eqref{e:SD3} and employing Lemma~\ref{lem:IFT}, we find that the dynamics of~\eqref{e:SD3} is equivalent to
\begin{align}
\partial_\xi V_n &= \widetilde{L}_{c^\ast} V_n + G(V_n;D,c), \label{e:SD4}\\
\partial_\xi D &= 0, \label{e:SD41}
\end{align}
where $G \colon \mathcal U_2 \times \mathcal Z_2 \times \mathcal V \to \mathcal{D}_{\neq 0}$ is given by
\begin{align*}
G(V_n;D,c) &= \widetilde{\pi}_{\neq 0} \left(\left(\widetilde{L}_{c} - \widetilde{L}_{c^\ast}\right)\left(V_n + G_1(V_n;D,c)V_1 + \psi\left(V_n + G_1(V_n;D,c)V_1;c\right)\right)\right.\\
&\left. \qquad \qquad + \, \frac{1}{c^2} \widetilde{N}\left(V_n + G_1(V_n;D,c)V_1 + \psi\left(V_n + G_1(V_n;D,c)V_1;c\right)\right)\right).
\end{align*}
and where $\widetilde{\pi}_{\neq 0}$ is the spectral projection of $\widetilde{L}_{c^\ast}$ associated with its off-zero neutral eigenvalues $\pm ik_0$, cf.~Lemma~\ref{lem:centralSpec}. Clearly, it holds $G(0;0,c^\ast) = 0$ and $\partial_{V_n} G(0;0,c^\ast) = 0$ by Theorem~\ref{theo:CM} and Lemma~\ref{lem:IFT}.
	Using Remark~\ref{rem:PhiAddProp} and~\eqref{e:IFT1}, we additionally find that $G(0;0,c) = 0$ for all $c \in \mathcal{V}$.
	Moreover, one readily verifies that~\eqref{e:SD4} is reversible under the symmetry $R$, since the same holds for system~\eqref{e:SD3} and $R$ commutes with $\psi(\cdot;c)$ by Theorem~\ref{theo:CM}.

We emphasize that we have reduced the $5$-dimensional system~\eqref{e:SD3}, which governs the dynamics on the center manifold, to the $4$-dimensional system~\eqref{e:SD4}-\eqref{e:SD41} treating $D$ as a (small) parameter.

\subsection{Normal form theorem} \label{sec42}

To simplify the analysis of the remaining $4$-dimensional reduced system~\eqref{e:SD4} further, we bring it into its normal form. That is, we look for a polynomial, near identity change of variables $V_n = \widetilde{V}_n + P_{c,D}(\widetilde{V}_n)$ such that~\eqref{e:SD4} transforms into $\partial_\xi \widetilde{V}_n = \widetilde{L}_{c^\ast} \widetilde{V}_n + \Phi_{c,D}(\widetilde{V}_n) + h.o.t.$, where $\Phi_{c,D}$ is a polynomial of degree $m$ that commutes with the semigroup $e^{\widetilde{L}_{c^\ast} \xi}$. The normal form of~\eqref{e:SD4} is considerably simpler as all non-resonant nonlinear terms up to order $m$ vanish due to the commuting property. We state the parameter-dependent normal form theorem, cf.~\cite[Theorem~2.2]{HaIo11}.

\begin{theorem} \label{theo:NF}
Let $\mathcal U_2,\mathcal V,\mathcal Z_2$ be as in Lemma~\ref{lem:IFT}. For any $(c,D) \in \mathcal V \times \mathcal Z$, there exists a polynomial $P_{c,D} \colon \mathcal{D}_{\neq 0} \to \mathcal{D}_{\neq 0}$ of degree $m$ such that
\begin{itemize}
\item The coefficients of $P_{c,D}$ are $C^m$ in $(c,D)$ and it holds
$$P_{c^\ast,0}(0) = 0, \qquad P_{c^\ast,0}'(0) = 0.$$
\item The change of variables
$$V_n = \widetilde{V}_n + P_{c,D}(\widetilde{V}_n),$$
transforms~\eqref{e:SD4} into
\begin{align} \label{e:SD51}
\partial_\xi \widetilde{V}_n = \widetilde{L}_{c^\ast} \widetilde{V}_n + \Phi_{c,D}(\widetilde{V}_n) + \rho(\widetilde{V}_n;c,D).
\end{align}
where $\Phi_{c,D} \colon \mathcal{D}_{\neq 0} \to \mathcal{D}_{\neq 0}$ is a polynomial of degree $m$ with $C^m$-coefficients in $(c,D)$ satisfying
\begin{align}
\Phi_{c^\ast,0}(0) = 0, \qquad \Phi_{c^\ast,0}'(0) = 0, \qquad \Phi_{c,D}\left(e^{\widetilde{L}_{c^\ast}\xi} V\right) = e^{\widetilde{L}_{c^\ast}\xi}\Phi_{c,D}\left(V\right), \label{e:EQ2}
\end{align}
for all $\xi \in \R, V \in \mathcal{D}_{\neq 0}$ and $(c,D) \in \mathcal V \times \mathcal Z_2$, and where we have $\rho \in C^{2m}(\mathcal U_2 \times \mathcal V \times \mathcal Z_2,\mathcal{D}_{\neq 0})$. In addition, there exists a constant $C > 0$ such that $\|\rho(\widetilde{V}_n;c,D)\| \leq C\|\widetilde{V}_n\|^m$ for all $(\widetilde{V}_n;c,D) \in \mathcal U_2 \times \mathcal V \times \mathcal Z_2$.
\item System~\eqref{e:SD51} is reversible under the symmetry $R$, and $P_{c,D}$ commutes with $R$.
\end{itemize}
\end{theorem}

\subsection{Normal form computation} \label{sec43}

In this section, we compute the relevant coefficients of the normal form~\eqref{e:SD51} for the construction of homoclinic solutions.

Since we are interested in real solutions $\widetilde{V}_n(\xi)$ to~\eqref{e:SD51} in $\mathcal{D}_{\neq 0}$, we write
\begin{align} \label{def:AB} \widetilde{V}_n(\xi) = A(\xi) V_2 + B(\xi) V_3 + \overline{A(\xi) V_2} + \overline{B(\xi) V_3},\end{align}
for coefficient functions $A,B\colon \R \to \C$. Substituting this expression into the normal form~\eqref{e:SD51}, noting its reversibility and exploiting the commuting property, cf.~\eqref{e:EQ2}, one can compute which nonlinear terms arise in $\Phi_{c,D}$ as a polynomial in $A,\overline{A},B$ and $\overline{B}$. This computation for $m = 3$ is performed in~\cite[Appendix~2]{KiIo00}, which brings us to the system
\begin{align} \label{e:SD5}
\begin{split}
\partial_\xi A &= ik_0 A + B + iA \mathcal{R}_{c,D}\left(|A|^2,i\left(A\overline{B}-\overline{A}B\right)\right)\\
&\qquad +\, \mathcal{O}\!\left(\|\widetilde{V}_n\|^4 + \|\widetilde{V}_n\||D|\left(\|\widetilde{V}_n\|^2 + |D|^2\right)\right),\\
\partial_\xi B &= ik_0 B + iB \mathcal{R}_{c,D}\left(|A|^2,i\left(A\overline{B}-\overline{A}B\right)\right) + A \mathcal{S}_{c,D}\left(|A|^2,i\left(A\overline{B}-\overline{A}B\right)\right)\\
&\qquad +\, \mathcal{O}\!\left(\|\widetilde{V}_n\|^4 + \|\widetilde{V}_n\||D|\left(\|\widetilde{V}_n\|^2 + |D|^2\right)\right),\\
\end{split}
\end{align}
where $\mathcal{R}_{c,D}$ and $\mathcal{S}_{c,D}$ are affine linear functions, with $C^m$-coefficients in $c,D$, which can be expanded as
\begin{align*}
\mathcal{R}_{c,D}\left(I,J\right) &= p_0(c) + rI + fJ + \mathcal{O}\!\left(|D| + |c-c^\ast|\left(|I| + |J|\right)\right),\\
\mathcal{S}_{c,D}\left(I,J\right) &= s_0(c) + sI + gJ + \mathcal{O}\!\left(|D| + |c-c^\ast|\left(|I| + |J|\right)\right),
\end{align*}
with $r,s,f,g \in \R$ and $p_0,s_0$ as in Corollary~\ref{cor:centralSpec}. By Theorem~\ref{theo:NF} system~\eqref{e:SD5} is reversible under the symmetry $\widetilde{R} \colon (A,B) \mapsto (\overline{A},-\overline{B})$.

For the construction of homoclinic solutions it is crucial to compute the sign of the coefficient $s$ in $\mathcal{S}_{c,D}$. We proceed as in~\cite{KiIo00,jamesSire05}. We relate the reduced system~\eqref{e:SD5} back to the original spatial dynamics formulation~\eqref{eq:spatDynCast}. First, we use Theorems~\ref{theo:CM} and~\ref{theo:NF}, Lemma~\ref{lem:IFT} and the decompositions~\eqref{def:q},~\eqref{def:d},~\eqref{def:Vn} and~\eqref{def:AB} to write a solution $V(\xi)$ to~\eqref{eq:spatDynCast} as
\begin{align*}
V(\xi) &= \tau(\xi)V_0 + d(\xi)V_1 + V_n(\xi) + \psi\left(d(\xi)V_1 + V_n(\xi);c\right)\\
 &= \tau(\xi)V_0 + D V_1 + \widetilde{V}_n(\xi) + P_{c,D}\left(\widetilde{V}_n(\xi)\right) + \left(G_1\left(\widetilde{V}_n(\xi) + P_{c,D}\left(\widetilde{V}_n(\xi)\right);D,c\right)-D\right)V_1\\
 &\qquad +\, \psi\left(\widetilde{V}_n(\xi) + P_{c,D}\left(\widetilde{V}_n(\xi)\right) + G_1\left(\widetilde{V}_n(\xi) + P_{c,D}\left(\widetilde{V}_n(\xi)\right);D,c\right)V_1;c\right)\\
 &= \tau(\xi)V_0 + D V_1 + A(\xi)V_2 + B(\xi) V_3 + \overline{A(\xi)V_2} + \overline{B(\xi)V_3} + G_2(A(\xi),B(\xi);D,c),
\end{align*}
where $A(\xi)$ and $B(\xi)$ solve~\eqref{e:SD5}, $D = I_1(V(\xi);c)$ denotes the first integral and the nonlinear remainder $G_2 \in C^m(\mathcal{U}_3,\widetilde{\mathcal{H}}_{c^\ast})$ satisfies $G_2(0,0;0,c^\ast) = 0$ and $\partial_{(A,B)} G_2(0,0;0,c^\ast) = 0$ and is explicitly given by
\begin{align*}
&G_2(A,B;D,c) = P_{c,D}\left(AV_2 + BV_3 + \overline{AV_2} + \overline{BV_3}\right)\\
&\qquad +\, \left(G_1\left(AV_2 + BV_3 + \overline{AV_2} + \overline{BV_3} + P_{c,D}\left(AV_2 + BV_3 + \overline{AV_2} + \overline{BV_3}\right);D,c\right)-D\right)V_1\\
&\qquad +\, \psi\left(AV_2 + BV_3 + \overline{AV_2} + \overline{BV_3} + P_{c,D}\left(AV_2 + BV_3 + \overline{AV_2} + \overline{BV_3}\right)\right. \\
&\qquad\qquad + \left.G_1\left(AV_2 + BV_3 + \overline{AV_2} + \overline{BV_3} + P_{c,D}\left(AV_2 + BV_3 + \overline{AV_2} + \overline{BV_3}\right);D,c\right)V_1;c\right),
 \end{align*}
where $\mathcal{U}_3 \subset \C^2 \times \R^2$ is a neighborhood of $(0,0;0,c^\ast)$.

Now to compute the coefficient $s$, we substitute the above expression for $V(\xi)$ into~\eqref{eq:spatDynCast} and use that $A(\xi)$ and $B(\xi)$ solve~\eqref{e:SD5}. Since $G_2$ is $C^m$, we can equate expressions at equal powers of $A,B,c-c^\ast$ and $D$ up to order $m = 3$. Thus, we consider the third-order Taylor expansion
\begin{align*}
\sum_{2 \leq s_1+s_2+s_3+s_4 \leq 3} \, \sum_{0 \leq s_5, s_6 \leq 3}A^{s_1}\overline{A}^{s_2} B^{s_3}\overline{B}^{s_4} D^{s_5} (c-c^\ast)^{s_6} \Phi_{s_1s_2s_3s_4s_5s_6},
\end{align*}
of $G_2(A,B;D,c)$ with coefficients $\Phi_{s_1s_2s_3s_4s_5s_6} \in \R^4$ , where we recall $G_2(0,0;0,c^\ast) = 0$ and $\partial_{(A,B)} G_2(0,0;0,c^\ast) = 0$. Reversibility of system~\eqref{e:SD5} yields
\begin{align*}
R \Phi_{s_1s_2s_3s_4s_5s_6} = \overline{\Phi_{s_2s_1s_4s_3s_5s_6}}.
\end{align*}

Matching coefficients at order $A^2{\overline{A}}^0 B^0\overline{B}^0 D^0 (c-c^\ast)^0$, at order $A^1\overline{A}^1 B^0\overline{B}^0 D^0 (c-c^\ast)^0$ and at order $A^0\overline{A}^1 B^1\overline{B}^0 D^0 (c-c^\ast)^0$ yields the equations
\begin{align} \label{e:coef}
\begin{split}
\left(2ik_0 - \widetilde{L}_{c^\ast}\right) \Phi_{200000} &= M_2(V_2,V_2) ,\\
-\widetilde{L}_{c^\ast} \Phi_{110000} &= 2M_2\left(V_2,\overline{V_2}\right) ,\\
-\widetilde{L}_{c^\ast} \Phi_{011000} &= -\Phi_{110000} + 2M_2\left(\overline{V_2},V_3\right) ,
\end{split}
\end{align}
respectively, with
\begin{align*}
M_2(Y,Z) &= \left(c^\ast\right)^{-2}\left(\begin{smallmatrix} 0 \\ 1 \\ 0\end{smallmatrix}\right)\left[a_1\left((Y_3(1)-Y_3(0))(Z_3(1)-Z_3(0))-(Y_3(0)-Y_3(-1))(Z_3(0)-Z_3(-1))\right)\right.\\
&\qquad\left. +\, a_2\left((Y_3(2)-Y_3(0))(Z_3(2)-Z_3(0))-(Y_3(0)-Y_3(-2))(Z_3(0)-Z_3(-2))\right)\right],
\end{align*}
cf.~the expansions~\eqref{expWs}. System~\eqref{e:coef} is most easily solved by first solving the ordinary differential equations arising in the last components of its three equations. Substituting the solutions of these ODEs back into~\eqref{e:coef} then yields a six-dimensional linear problem. All in all, one finds
\begin{align*}
\Phi_{200000} &= M_2(V_2,V_2)\left(\frac{1}{\sigfunk(2ik_0;c^\ast)} \begin{pmatrix} 1 \\ e^{2ik_0} \\ p \mapsto e^{2ik_0p}\end{pmatrix} + \frac{1}{4k_0^2\left((c^\ast)^2 - 1\right)}V_0\right), \\ \Phi_{110000} &= \frac{2M_2(V_1,V_2)}{1-(c^\ast)^2} V_1,
\end{align*}
whereas $\Phi_{011000}$ is only determined up to addition of a scalar multiple of $V_1$ (recall $\ker(\widetilde{L}_{c^\ast})$ is spanned by $V_1$). However, the precise value of $\Phi_{011000}$ is not relevant for the further analysis. Next, we match coefficients at order $A^2\overline{A}^1 B^0\overline{B}^0 D^0 (c-c^\ast)^0$ and find
\begin{align} \label{e:coef2}
\begin{split}
\left(ik_0 - \widetilde{L}_{c^\ast}\right) \Phi_{210000} &= -irV_2 - sV_3 + 2M_2(V_2,\Phi_{110000}) + 2M_2\left(\overline{V_2},\Phi_{200000}\right) + 3M_3\left(V_2,V_2,\overline{V}_2\right),
\end{split}
\end{align}
 with
\begin{align*}
M_3(Y,W,Z) :=& \left(c^\ast\right)^{-2}\left(\begin{smallmatrix} 0 \\ 1 \\ 0\end{smallmatrix}\right)\left[b_1\left((Y_3(1)-Y_3(0))(W_3(1)-W_3(0))(Z_3(1)-Z_3(0))\right.\right.\\
&\phantom{\left(c^\ast\right)^{-2}\left(\begin{smallmatrix} 0 \\ 1 \\ 0\end{smallmatrix}\right)}\left. - \,(Y_3(0)-Y_3(-1))(W_3(0)-W_3(-1))(Z_3(0)-Z_3(-1))\right)\\
&\phantom{\left(c^\ast\right)^{-2}\left(\begin{smallmatrix} 0 \\ 1 \\ 0\end{smallmatrix}\right)}\left. +\, b_2\left((Y_3(2)-Y_3(0))(W_3(2)-W_3(0))(Z_3(2)-Z_3(0))\right.\right.\\
&\phantom{\left(c^\ast\right)^{-2}\left(\begin{smallmatrix} 0 \\ 1 \\ 0\end{smallmatrix}\right)}\left.\left.-\,(Y_3(0)-Y_3(-2))(W_3(0)-W_3(-2))(Z_3(0)-Z_3(-2))\right)\right].
\end{align*}
cf.~the expansions~\eqref{expWs}. Again we first solve the ODE arising in the last component of~\eqref{e:coef2}. Substituting the solution of this ODE into the other components of~\eqref{e:coef2} yields a two-dimensional linear problem, which has a solution provided the solvability condition
\begin{align}
s = \frac{2}{\partial_{\lambda}^2 \sigfunk(ik_0;c^\ast)} \left(4\frac{M_2(V_1,V_2)^2}{1-(c^\ast)^2} - 2\frac{M_2(V_2,V_2)^2}{\sigfunk(2ik_0;c^\ast)} + 3M_3\left(V_2,V_2,\overline{V_2}\right)\right), \label{def:s}
\end{align}
is satisfied. The solvability condition~\eqref{def:s} determines the coefficient $s$ explicitly. Indeed, we recall that the determinantal function $\sigfunk$ is explicitly given by~\eqref{eq:linDispersion} and one readily computes
\begin{align}
\label{def:Ms}
\begin{split}
M_2(V_1,V_2) &= 2\left(a_1 \left(\cos(k_0) - 1\right) + 2a_2 \left(\cos(2k_0) - 1\right)\right) \in \R,\\
M_2(V_2,V_2) &= 4i\left(a_1 \sin(k_0)\left(\cos(k_0) - 1\right) + a_2\sin(2k_0)\left(\cos(2k_0) - 1\right)\right) \in i\R,\\
M_3\left(V_2,V_2,\overline{V_2}\right) &= -16\left(b_1 \sin^4\left(\frac{k_0}{2}\right) + b_2 \sin^4(k_0)\right) \in \R,
\end{split}
\end{align}
which confirms that $s$ lies in $\R$. By Lemma~\ref{lem:centralSpec} it holds
\begin{align*}
1-(c_\ast)^2 < 0, \qquad \partial_{\lambda}^2 \sigfunk(ik_0;c^\ast) > 0, \qquad \sigfunk(2ik_0;c^\ast) < 0.
\end{align*}
Thus, for any choice of $a_1,a_2 \in \R$ it holds
\begin{align*}
4\frac{M_2(V_1,V_2)^2}{1-(c^\ast)^2} - 2\frac{M_2(V_2,V_2)^2}{\sigfunk(2ik_0;c^\ast)} \leq 0.
\end{align*}
Hence, $s$ has negative sign as long as the explicit condition
\begin{align} \label{e:signcond}
4\frac{M_2(V_1,V_2)^2}{1-(c^\ast)^2} - 2\frac{M_2(V_2,V_2)^2}{\sigfunk(2ik_0;c^\ast)} < 48\left(b_1 \sin^4\left(\frac{k_0}{2}\right) + b_2 \sin^4(k_0)\right),
\end{align}
is satisfied, where we recall that the determinantal function $\sigfunk$ is explicitly given by~\eqref{eq:linDispersion} and $M_2(V_1,V_2)$ and $M_2(V_2,V_2)$ are computed in~\eqref{def:Ms}.

Note that~\eqref{e:signcond} is in particular satisfied if $b_1$ and $b_2$ are both positive.

\subsection{Homoclinic solutions and their persistence} \label{sec44}

In this section we construct homoclinic solutions to the spatial dynamics formulation~\eqref{e:SD5}. As outlined in~\S\ref{sec:approach}, we first establish explicit homoclinic solutions to the normal form~\eqref{e:SD5} by truncating at order 3 and setting $D = 0$. We then argue that these homoclinics persist when reintroducing the higher-order terms to~\eqref{e:SD5}, where we exploit the reversibility of the system.

Thus, truncating the normal form~\eqref{e:SD5} at order 3 and setting $D = 0$ gives
\begin{align} \label{e:SD52}
\begin{split}
\partial_\xi A &= ik_0 A + B + iA \mathcal{R}_{c,0}\left(|A|^2,i\left(A\overline{B}-\overline{A}B\right)\right),\\
\partial_\xi B &= ik_0 B + iB \mathcal{R}_{c,0}\left(|A|^2,i\left(A\overline{B}-\overline{A}B\right)\right) + A \mathcal{S}_{c,0}\left(|A|^2,i\left(A\overline{B}-\overline{A}B\right)\right).
\end{split}
\end{align}
Now take $a_1,a_2,b_1,b_2 \in \R$ in~\eqref{expWs} such that~\eqref{e:signcond} is satisfied and, thus, the coefficient $s$ of $\mathcal{S}_{c,0}$ has negative sign. Moreover, let $c - c^\ast > 0$ be sufficiently small. Then, as in~\cite{IoossJames05}, the truncated system~\eqref{e:SD52} admits a one-parameter family of solutions, which are homoclinic to $(0,0)$ and explicitly given by
\begin{align} \label{e:homsol}
A_\theta(\xi) &= r_0(\xi) e^{i(k_0\xi +\psi_0(\xi) + \theta)},\qquad B_\theta(\xi) = r_0'(\xi) e^{i(k_0\xi+\psi_0(\xi) + \theta)},
\end{align}
with $\theta \in \R$ and
\begin{align*}
r_0(\xi) = \sqrt{\frac{2s_0(c)}{-s}} \frac{1}{\cosh\left(\sqrt{s_0(c)}\xi\right)}, \qquad \psi_0(\xi) = p_0(c)\xi + 2\frac{r}{s} \sqrt{s_0(c)}\tanh\left(\sqrt{s_0(c)}\xi\right),
\end{align*}
where $s_0$ and $p_0$ are as in Corollary~\ref{cor:centralSpec}. We find that the homoclinic solutions $(A_\theta(\xi),B_\theta(\xi))$ have small supremum norms of order $\mathcal{O}(\sqrt{c-c^\ast})$. Moreover, one readily observes that $(A_\theta(\xi),B_\theta(\xi))$ is reversible if and only if the parameter $\theta$ is an integer multiple of $\pi$. Thus, precisely two members of the one-parameter family~\eqref{e:homsol} of homoclinic solutions are reversible.

We prove that these reversible homoclinics persist when reintroducing the higher-order terms to the normal form~\eqref{e:SD5}.

\begin{proposition}\label{prop:homoclinics}
There exists a constant $C > 0$ such that, provided $c - c^\ast > 0$ is sufficiently small, system~\eqref{e:SD5} admits for $D = 0$ two reversible solutions $(\widetilde{A}_0(\xi),\widetilde{B}_0(\xi))$ and $(\widetilde{A}_\pi(\xi),\widetilde{B}_\pi(\xi))$, which are homoclinic to $(0,0)$ and enjoy the estimate
\begin{align} \left|\widetilde{A}_{\theta^*}(\xi) - A_{\theta^*}(\xi)\right|, \left|\widetilde{B}_{\theta^*}(\xi) - B_{\theta^*}(\xi)\right| \leq C|c-c^\ast|, \label{e:persest}\end{align}
for ${\theta^*} = 0,\pi$, where $(A_\theta(\xi),B_\theta(\xi)), \, \theta \in \R$ is defined in~\eqref{e:homsol}.
\end{proposition}
\begin{proof}
A proof of this result can be found in~\cite[\S IV.3]{IP93}, albeit without deriving the estimate~\eqref{e:persest}. We follow the proof provided in~\cite{IP93} and show how the estimate~\eqref{e:persest} arises. Throughout we use the abbreviation $\epsilon = \sqrt{c-c^\ast}$.

First, we introduce the new real variable $X = \epsilon^{-1} (\mathrm{Re}(A),\mathrm{Im}(A),\mathrm{Re}(B),\mathrm{Im}(B))^\top \in \R^4$, in which system~\eqref{e:SD5} at $D = 0$ takes the abstract form
\begin{align} \partial_\xi X = F_1(\epsilon)X + \epsilon^2 F_2(X;\epsilon) + \epsilon^3 R(X;\epsilon), \label{e:SD53}\end{align}
and in which the truncated system~\eqref{e:SD52} reads
\begin{align} \partial_\xi X = F_1(\epsilon)X + \epsilon^2 F_2(X;\epsilon), \label{e:trunc}\end{align}
with $F_2(X;\epsilon)$ cubic in $X$ and
\begin{align*} F_1(\epsilon) = \begin{pmatrix} 0 & -k_0 - p_0\big(c^\ast + \epsilon^2\big) & 1 & 0 \\ k_0 + p_0\big(c^\ast + \epsilon^2\big) & 0 & 0 & 1 \\ s_0\big(c^\ast + \epsilon^2\big) & 0 & 0 & -k_0 - p_0\big(c^\ast + \epsilon^2\big) \\ 0 & s_0\big(c^\ast + \epsilon^2\big) & k_0+p_0\big(c^\ast + \epsilon^2\big) & 0 \end{pmatrix} \end{align*}
where $p_0, s_0 : \curlV \rightarrow \R$ are defined in Corollary~\ref{cor:centralSpec}.
It follows from Theorem~\ref{theo:NF} that $R(X;\epsilon)$ is defined on a neighborhood of the origin in $\R^4 \times \R$ and is $C^{2m}$ in $X$ and $\epsilon$ with $m = 3$. In addition, the cubic form $F_2(X;\epsilon)$ also has $C^{2m}$ coefficients in $\epsilon$. Finally, both~\eqref{e:SD53} and~\eqref{e:trunc} are reversible under the symmetry $\varrho : (u,v,w,z)^\top \mapsto (u,-v,-w,z)^\top$ on $\R^4$.

Since the four eigenvalues $\pm i(k_0+p_0(c^\ast + \epsilon^2)) \pm s_0(c^\ast + \epsilon^2)$ of $F_1(\epsilon)$ have non-zero real part for $\epsilon > 0$, cf.~Corollary~\ref{cor:centralSpec}, the fixed point $0$ in systems~\eqref{e:SD53} and~\eqref{e:trunc} is hyperbolic with two-dimensional stable manifold. The homoclinic solutions~\eqref{e:homsol} to the truncated system~\eqref{e:SD52} correspond to homoclinic solutions $X_\epsilon(\xi;\theta)$ to system~\eqref{e:trunc} with $\theta \in \R$. Thus, $M_\epsilon = \{X_\epsilon(\epsilon^{-1} \xi;\theta) : \xi \in \R, \theta \in \R\}$ is a two-dimensional submanifold of the stable manifold of the fixed point $0$ in~\eqref{e:trunc}. Using the explicit expressions provided in~\eqref{e:homsol} and Corollary~\ref{cor:centralSpec}, one readily observes that $M_\epsilon$ depends smoothly on $\epsilon$ and is a well-defined two-dimensional manifold in the limit $\epsilon \to 0$. Clearly, the manifold $M_\epsilon$ intersects the reversibility plane $\ker(I-\varrho)$ at the points $X_\epsilon(0;0)$ and $X_\epsilon(0;\pi)$. The tangent vectors $\partial_\theta X_\epsilon(0;\theta^*)$ and $\partial_\xi X_\epsilon(0;\theta^*)$ at the intersection points can be computed explicitly using~\eqref{e:homsol} for $\theta^* = 0,\pi$. One verifies that the limiting vectors $\lim_{\epsilon \downarrow 0} \partial_\theta X_\epsilon(0;\theta^*)$ and $\lim_{\epsilon \downarrow 0} \partial_\xi X_\epsilon(0;\theta^*)$ span a space complementary to $\ker(I-\varrho)$. Therefore, the two intersections of $M_0$ and $\ker(I-\varrho)$ are transversal, and thus so are the intersection between $M_\epsilon$ and $\ker(I-\varrho)$ for $\epsilon > 0$ sufficiently small.

Now take $\theta \in \R$. We consider the variational equation of~\eqref{e:trunc} about $X_\epsilon(\xi;\theta)$ given by
\begin{align} \partial_\xi X = \left(F_1(\epsilon) + \epsilon^2 F_2'(X_\epsilon(\xi;\theta);\epsilon)\right) X. \label{e:var1}\end{align}
We transform~\eqref{e:var1} by setting $\tau = \epsilon \xi$ and $Y = \mathcal{D}_\epsilon(\xi) X$, where $\mathcal{D}_\epsilon(\xi)$ is the diagonal matrix
$$\mathcal{D}_\epsilon(\xi) = \mathrm{diag}\left(e^{-ik_0\xi},e^{ik_0\xi},\epsilon e^{-ik_0\xi},\epsilon e^{ik_0\xi}\right).$$
Exploiting the explicit structure of $F_2(X;\epsilon)$, cf.~\eqref{e:SD53}, we arrive at the linear system
\begin{align} \partial_\tau Y = \left(\widetilde{F}_1(\epsilon) + \epsilon \widetilde{F}_2(\tau;\epsilon,\theta)\right)Y, \label{e:var2}\end{align}
with
\begin{align*} \widetilde{F}_1(\epsilon) = \begin{pmatrix} 0 & -\epsilon^{-1} p_0\big(c^\ast + \epsilon^2\big) & 1 & 0 \\ \epsilon^{-1}p_0\big(c^\ast + \epsilon^2\big) & 0 & 0 & 1 \\ \epsilon^{-2} s_0\big(c^\ast + \epsilon^2\big) & 0 & 0 & -\epsilon^{-1} p_0\big(c^\ast + \epsilon^2\big) \\ 0 & \epsilon^{-2} s_0\big(c^\ast + \epsilon^2\big) & \epsilon^{-1} p_0\big(c^\ast + \epsilon^2\big) & 0 \end{pmatrix},\end{align*}
and $\widetilde{F}_2(\tau;\epsilon,\theta) = \epsilon \mathcal{D}_\epsilon(0)^{-1} F_2'(X_\epsilon(\epsilon^{-1} \tau;\theta);\epsilon) \mathcal{D}_\epsilon(0)$. First, the matrix $\widetilde{F}_1(\epsilon)$ is $\epsilon$-uniformly bounded and the matrix function $\widetilde{F}_2(\tau;\epsilon)$ is $\epsilon$-uniformly bounded on $\R$. Secondly, the eigenvalues of $\widetilde{F}_1(\epsilon)$ are $\epsilon$-uniformly bounded away from the imaginary axis. Thirdly, $\widetilde{F}_2(\tau;\epsilon,\theta)$ converges exponentially to $0$ as $\tau \to \infty$ with $\epsilon$-independent rate. Combining these three assertions yields that system~\eqref{e:var2} has an exponential dichotomy on $\R_+$ with $\epsilon$-independent constants $K,\mu > 0$, cf.~\cite[Lemma~3.4]{PA84}. Hence, undoing the coordinate transform, we establish an exponential dichotomy on $\R_+$ with constants $\frac{K}{\epsilon}, \epsilon \mu > 0$ for system~\eqref{e:var1}. We denote by $P(\xi;\epsilon,\theta)$ the associated projection matrix.

Inserting the perturbative ansatz $X = X_\epsilon(\xi;\theta) + Z$ into~\eqref{e:SD53}, we find that $Z$ satisfies
\begin{align}
\begin{split}
\partial_\xi Z &= \left(F_1(\epsilon) + \epsilon^2 F_2'(X_\epsilon(\xi;\theta);\epsilon)\right)Z + N(\xi,Z;\epsilon,\theta),
\end{split} \label{e:SD54}\end{align}
with
\begin{align*}
N(\xi,Z;\epsilon,\theta) &= \epsilon^2 \left(F_2(X_\epsilon(\xi;\theta) + Z;\epsilon) - F_2(X_\epsilon(\xi;\theta);\epsilon) - F_2'(X_\epsilon(\xi;\theta);\epsilon)Z\right)\\
&\qquad + \, \epsilon^3 R(X_\epsilon(\xi;\theta) + Z;\epsilon),
\end{align*}
We construct a bounded solution to system~\eqref{e:SD54} by solving the associated integral equation
\begin{align}
\begin{split}
Z(\xi) &= \int_0^\xi P(\xi;\epsilon,\theta)\Phi(\xi,\zeta;\epsilon,\theta)N(\zeta,Z;\epsilon,\theta) d \zeta\\ &\qquad - \, \int_\xi^\infty \left(I-P(\xi;\epsilon,\theta)\right)\Phi(\xi,\zeta;\epsilon,\theta)N(\zeta,Z;\epsilon,\theta) d \zeta,
\end{split} \label{e:SD55}\end{align}
where $\Phi(\xi,\zeta;\epsilon,\theta)$ denotes the evolution of the variational equation~\eqref{e:var1}. The right-hand side of~\eqref{e:SD55} defines a nonlinear map $\mathcal{F}_{\epsilon,\theta}$ on an $\epsilon$-independent neighborhood of the origin in the Banach space $C_b(\R_+,\R^4)$ of bounded and continuous functions endowed with the supremum norm. Due to the exponential dichotomy of system~\eqref{e:var1} there exists an $\epsilon$-independent constant $C > 0$ such that
\begin{align*} \left\|\mathcal{F}_{\epsilon,\theta}(Z)\right\|_\infty &\leq C\left(\epsilon + \|Z\|_\infty^2\right), \\
\left\|\mathcal{F}_{\epsilon,\theta}(Z) - \mathcal{F}_{\epsilon,\theta}(W)\right\|_\infty &\leq C\left(\epsilon + \|Z\|_\infty + \|W\|_\infty\right)\|Z-W\|_\infty.
\end{align*}
Hence, setting $\rho_0 = C+1$ and taking $\epsilon > 0$ sufficiently small, $\mathcal{F}_{\epsilon,\theta}$ defines a contraction mapping on a ball of radius $\rho_0\epsilon$ in $C_b(\R_+,\R^4)$. So, there exists a unique fixed point $Z_\epsilon(\cdot;\theta) \in C_b(\R_+,\R^4)$ of $\mathcal{F}_{\epsilon,\theta}$. Clearly, it holds $\|Z_\epsilon(\cdot;\theta)\|_\infty \leq \rho_0\epsilon$. 
Using the exponential dichotomy again, one verifies with the aid of~\eqref{e:SD55} that $Z_\epsilon(\xi;\theta)$ converges to $0$ as $\xi \to \infty$.

Thus, $\widetilde{X}_\epsilon(\xi;\theta) = X_\epsilon(\xi;\theta) + Z_\epsilon(\xi;\theta)$ is a solution to~\eqref{e:SD53} converging to $0$ as $\xi \to \infty$. In particular, $\widetilde{M}_\epsilon = \{\widetilde{X}_\epsilon(\epsilon^{-1} \xi;\theta) : \xi \in \R, \theta \in \R\}$ is a two-dimensional submanifold of the stable manifold of the hyperbolic fixed point $0$ in~\eqref{e:SD53}. It follows from $\|Z_\epsilon(\cdot;\theta)\|_\infty \leq \rho_0\epsilon$ that $\widetilde{M}_\epsilon$ lies $\mathcal{O}(\epsilon)$-close to $M_\epsilon$, and thus to $M_0$, and must therefore intersect the reversibility plane $\ker(I-\varrho)$ at two points $\widetilde{X}_\epsilon(\xi_1;\theta_1)$ and $\widetilde{X}_\epsilon(\xi_2;\theta_2)$ which lie $\mathcal{O}(\epsilon)$-close to the transversal intersection points $\lim_{\epsilon \downarrow 0} X_\epsilon(0;0)$ and $\lim_{\epsilon \downarrow 0} X_\epsilon(0;\pi)$ of $M_0$ with $\ker(I-\varrho)$. In particular, it holds $|\xi_1|,|\xi_2| \leq C\epsilon^2$ and $|\theta_1|, |\theta_2 - \pi| \leq C\epsilon$ for some constant $C > 0$.

We conclude that $\widetilde{X}_\epsilon(\xi-\xi_1;\theta_1)$ and $\widetilde{X}_\epsilon(\xi-\xi_2;\theta_2)$ are the desired reversible homoclinic solutions to~\eqref{e:SD53}. The estimate~\eqref{e:persest} follows from $|\xi_1|,|\xi_2| \leq C\epsilon^2$, $|\theta_1|, |\theta_2 - \pi| \leq C\epsilon$ and $\|Z_\epsilon(\cdot;\theta)\|_\infty \leq \rho_0\epsilon$ upon recalling that~\eqref{e:SD53} is a rescaled version of~\eqref{e:SD5} at $D = 0$.
\end{proof}

\section{Proof of Theorem \texorpdfstring{\ref{theo:mainresult}}{ }}\label{sec:proofMainThm}

Using the results in~\S\ref{sec2}--\S\ref{sec4} we now prove our main result, Theorem~\ref{theo:mainresult}.
We split the proof into three parts.
First, we construct a homoclinic solution $W_\text{hom}$ to~\eqref{e:SD1} starting from the homoclinic orbits obtained in Proposition~\ref{prop:homoclinics}. In detail, we show that
\begin{align*}
	W_\text{hom}(\xi) = A_0(\xi) V_2 + B_0(\xi) V_3 + c.c. + \curlR(\xi),
\end{align*}
where $A_0, B_0$ are given in~\eqref{e:homsol} with $\theta = 0$ and the remainder $\curlR \in L^\infty(\R)$ satisfies $\norm[\infty]{\curlR} = \curlO(\snorm{c-c^\ast})$ and $\lim_{\snorm{\xi} \rightarrow \infty} \curlR(\xi) = 0$.
Subsequently, we solve~\eqref{e:SD2} and prove that the solution $\tau_\text{hom}$ corresponding to $W_\text{hom}$ satisfies $\norm[\infty]{\tau_\text{hom}} = \curlO(\snorm{c - c^\ast})$ and $\lim_{\xi \rightarrow \pm \infty} \tau(\xi) = \tau_\pm \in \R$.
In particular, we show that, if the potentials $\curlW_1$, $\curlW_2$ are symmetric, then $\tau_+ = \tau_-$, which proves Theorem~\ref{thm:trueHomoclinic}.
Finally, combining the first two steps, we obtain a solution to~\eqref{eq:spatDynCast} and show that this solution satisfies the NLS approximation~\eqref{eq:NLSestimate}.

\begin{remark}[Notation]
	Throughout this section, we use the following notation.
	As in the proof of Proposition~\ref{prop:homoclinics}, we abbreviate $\varepsilon = \sqrt{c - c^\ast}$, which is well-defined since $c > c^\ast$.
	Furthermore, we denote a generic $\varepsilon$-independent constant by $C$ and a generic remainder term by $\curlR(\xi)$ if it satisfies $\curlR \in L^\infty(\R)$ with
	\begin{align*}
		\norm[\infty]{\curlR(\xi)} = \curlO(\varepsilon^2),
	\end{align*}
	and is exponentially localized.
\end{remark}

\begin{remark}
	Recall that Proposition~\ref{prop:homoclinics} provides two reversible homoclinic orbits, one for ${\theta^*} = 0$ and one for ${\theta^*} = \pi$.
	Although the calculations in this section are done for ${\theta^*} = 0$, they can be done analogously for ${\theta^*} = \pi$.
\end{remark}

\subsection{Constructing a homoclinic solution to \texorpdfstring{\eqref{e:SD1}}{ }}

By Proposition~\ref{prop:homoclinics} there exists a reversible homoclinic solution $(\widetilde{A}_0(\xi),\widetilde{B}_0(\xi))$ to~\eqref{e:SD5}, which is $\mathcal{O}(\varepsilon^2)$-close in $L^\infty$-norm to the explicit homoclinic solution $(A_0(\xi), B_0(\xi))$ given in~\eqref{e:homsol}.
Using~\eqref{def:AB}, we therefore find that~\eqref{e:SD51} has a solution
\begin{align*}
	\widetilde{V}_{n,\text{hom}}(\xi) := A_0(\xi) V_2 + B_0(\xi) V_3 + c.c. + \curlR(\xi).
\end{align*}
Recalling that $D = 0$ in Proposition~\ref{prop:homoclinics} we invert the normal form transformation in Theorem~\ref{theo:NF} and find that $V_{n,\text{hom}} = \widetilde{V}_{n,\text{hom}} + P_{c,0}(\widetilde{V}_{n,\text{hom}})$ is a solution to~\eqref{e:SD4}.

We now show that $V_{n,\text{hom}}$ is a homoclinic to zero by proving that $P_{c,0}(0) = 0$ for $c \in \widetilde{\curlV} \subset \curlV$, see Theorem~\ref{theo:CM}.
Using~\cite[Remark 3.2.3]{HaIo11}, we have $\Phi_{c,D}(0) \in \operatorname{ker}(\tildeL) = \operatorname{span}(V_1)$.
Since $\Phi_{c,D}$ maps into $\curlD_{\neq 0}$ we therefore obtain $\Phi_{c,D}(0) = 0$.
Then, differentiating $V_n = \widetilde{V}_n + P_{c,D}(\widetilde{V}_n)$ with respect to $\xi$ and evaluating the resulting equation at $\widetilde{V}_n = 0$ yields
\begin{align*}
	\tildeL P_{c,0}(0) + G(P_{c,0}(0);0,c) = 0,
\end{align*}
with $G$ from~\eqref{e:SD4}.
Using that $\tildeL$ is invertible on $\curlD_{\neq 0}$ this equation uniquely defines $P_{c,0}(0)$ for $c \in \widetilde{\curlV} \subset \curlV$ via the implicit function theorem.
Since $G(0;0,c) = 0$, this yields $P_{c,0}(0) = 0$.
Therefore, $V_{n,\text{hom}}$ is a homoclinic to zero.
Additionally using that $P^\prime_{c^\ast,0}(0) = 0$ and the fact that the coefficients of $P_{c,0}$ are $C^m$ in $\varepsilon^2 = c-c^\ast$ we have the estimate
\begin{align*}
	\snorm{P_{c,0}(\widetilde{V}_n)} \leq C \left(\snorm{\widetilde{V}_n}^2 + \varepsilon^2 \snorm{\widetilde{V}_n}\right).
\end{align*}
This yields that $V_{n,\text{hom}}$ is of the form
\begin{align*}
	V_{n,\text{hom}}(\xi) = A_0(\xi) V_2 + B_0(\xi) V_3 + c.c. + \curlR(\xi).
\end{align*}

Choosing $\varepsilon > 0$ sufficiently small, we find $V_{n,\text{hom}}(\xi) \in \curlU_2$ for all $\xi \in \R$.
Hence, we can apply Lemma~\ref{lem:IFT} and the center manifold theorem~\ref{theo:CM} to obtain a solution to~\eqref{e:SD1} given by
\begin{align*}
	W_{\text{hom}}(\xi) = V_{n,\text{hom}}(\xi) + G_1(V_{n,\text{hom}}(\xi);0,c) V_1 + \psi\left(V_{n,\text{hom}}(\xi) + G_1(V_{n,\text{hom}}(\xi);0,c) V_1;c\right).
\end{align*}
Using~\eqref{e:IFT1} and Remark~\ref{rem:PhiAddProp}, we find that $W_{\text{hom}}$ is a homoclinic to zero and satisfies
\begin{align}
	W_{\text{hom}}(\xi) = A_0(\xi) V_2 + B_0(\xi) V_3 + c.c. + \curlR(\xi).
	\label{eq:Whom}
\end{align}
This concludes the first part of the proof of Theorem~\ref{theo:mainresult}.

\subsection{Solving the shift equation} 

We now solve the shift equation~\eqref{e:SD2} corresponding to $W_\text{hom}$, that is,
\begin{align}
	\partial_\xi \tau = \chi_{1,c^\ast}^\ast(W_\text{hom}).
	\label{eq:shiftEq}
\end{align}
Since $W_\text{hom}$ is exponentially localized and $\chi_{1,c^\ast}^\ast(\cdot)$ is continuous and satisfies $\chi_{1,c^\ast}^\ast(0) = 0$, we can integrate~\eqref{eq:shiftEq} and find for any $\tau_- \in \R$ that
\begin{align*}
	\tau_\text{hom}(\xi) = \tau_- + \int_{-\infty}^\xi \chi_{1,c^\ast}^\ast(W_\text{hom}(\tilde{\xi})) \,d\tilde{\xi},
\end{align*}
solves~\eqref{eq:shiftEq}. In particular, $\tau_\text{hom}(\xi)$ converges to some $\tau_+ \in \R$ as $\xi \rightarrow +\infty$ and is thus a front.
Furthermore, we find that $\snorm{\tau_+ - \tau_-} = \curlO(\varepsilon^2)$ by using the expression~\eqref{eq:Whom} and
\begin{align*}
	\chi_{1,c^\ast}^\ast(W_\text{hom}) = \chi_{1,c^\ast}^\ast(\curlR) = \curlO(\varepsilon^2),
\end{align*}
where the first equality is due to the fact that $\chi_{1,c^\ast}^\ast(V_2) = \chi_{1,c^\ast}^\ast(V_3) = 0$ and the second one follows from the linearity of $\chi_{1,c^\ast}^\ast$.
This in particular yields for the choice $\tau_- = 0$ that $\tau_\text{hom} = \curlO(\varepsilon^2)$.

While $\tau_- \in \R$ can be chosen arbitrary, in general we cannot expect that $\tau_+ = \tau_-$, that is, $\tau_\text{hom}$ is in general not a pulse.
However, if we additionally assume that the potentials $\curlW_1$, $\curlW_2$ in the FPU equation~\eqref{FPU} are symmetric, i.e., it holds $\curlW_j(r) = \curlW_j(-r)$ for all $r \in \R$, we obtain that $\tau_\text{hom}$ is a pulse, which follows from the next result.

\begin{proposition}
	Assume that $\curlW_1$, $\curlW_2$ in~\eqref{FPU} are symmetric and let
	\begin{align*}
		\tau_\text{hom}(\xi) = \int_{-\infty}^\xi \chi_{1,c^\ast}^\ast(W_\text{hom}(\tilde{\xi})) \,d\tilde{\xi},
	\end{align*}
	with $W_\text{hom}$ as in~\eqref{eq:Whom}.
	Then, it holds
	\begin{align*}
		\lim_{\xi \rightarrow +\infty} \tau_\text{hom}(\xi) = \int_{-\infty}^\infty \chi_{1,c^\ast}^\ast(W_\text{hom}(\tilde{\xi})) \,d\tilde{\xi} = 0.
	\end{align*}
\end{proposition}
\begin{proof}
	Since $\curlW_1$ and $\curlW_2$ are symmetric, the system~\eqref{spatdyn} gains the additional invariance $v(\xi) \mapsto v(-\xi)$.
	Therefore, the system~\eqref{eq:spatDynCast} is reversible under the symmetry
	\begin{align*}
		S \colon \curlH \rightarrow \curlH, \qquad S\begin{pmatrix}
			z \\ y \\ U
		\end{pmatrix} = \begin{pmatrix}
			z \\ - y \\ p \mapsto U(-p)
		\end{pmatrix},
	\end{align*}
	which can be verified by direct computation.
	Following the analysis in~\S\ref{sec3} and~\S\ref{sec4} this yields that the map $\psi$ in the center manifold theorem~\ref{theo:CM} and the polynomial $P_{c,D}$ in the normal form theorem~\ref{theo:NF} commute with $S$.
	
	We recall from Proposition~\ref{prop:homoclinics} that the homoclinic orbit $(\widetilde{A}_0(\xi), \widetilde{B}_0(\xi))$ is reversible, i.e.
	\begin{align*}
		\widetilde{R}(\widetilde{A}_0(\xi), \widetilde{B}_0(\xi)) = (\widetilde{A}_0(-\xi), \widetilde{B}_0(-\xi)),
	\end{align*}
	where $\widetilde{R}(A,B) = (\overline{A},-\overline{B})$.
	Using that $SV_2 = \overline{V_2}$ and $SV_3 = -\overline{V_3}$ we find that $\widetilde{V}_{n,\text{hom}} = \widetilde{A}_0V_2 + \widetilde{B}_0 V_3 + c.c.$ satisfies
	\begin{align*}
		S\widetilde{V}_{n,\text{hom}}(\xi) = \widetilde{V}_{n,\text{hom}}(-\xi).
	\end{align*}
	Then, since $P_{c,D}$ commutes with $S$ it holds
	\begin{align*}
		SV_{n,\text{hom}}(\xi) = S(\widetilde{V}_{n,\text{hom}}(\xi) + P_{c,0}(\widetilde{V}_{n,\text{hom}}(\xi))) = \widetilde{V}_{n,\text{hom}}(-\xi) + P_{c,0}(\widetilde{V}_{n,\text{hom}}(-\xi))) = V_{n,\text{hom}}(-\xi).
	\end{align*}
	To obtain the desired statement, we recall from the proof of Lemma~\ref{lem:IFT} that $\chi_{1,c^\ast}^\ast(W_\text{hom}(\xi)) = d(\xi)$ is implicitly defined by
	\begin{align*}
		d(\xi) = \chi_{1,c^\ast}^\ast(V_{n,\text{hom}}(\xi) + d(\xi) V_1 + \psi(V_{n,\text{hom}}(\xi) + d(\xi) V_1),
	\end{align*}
	where we used that $D = 0$.
	Using $SV_1 = -V_1$ and that $\psi$ commutes with $S$ we then find
	\begin{align*}
		W_\text{hom}(-\xi) = S(V_{n,\text{hom}}(\xi) - d(-\xi) V_1 + \psi(V_{n,\text{hom}}(\xi) - d(-\xi) V_1)).
	\end{align*}
	Additionally observing that $\chi_{1,c^\ast}^\ast(SV) = -\chi_{1,c^\ast}^\ast(V)$, which follows from the definition of $\chi_{1,c^\ast}^\ast$, we find that $d(-\xi)$ satisfies
	\begin{align*}
		-d(-\xi) &= \chi_{1,c^\ast}^\ast (V_{n,\text{hom}}(\xi) - d(-\xi) V_1 + \psi(V_{n,\text{hom}}(\xi) - d(-\xi) V_1)).
	\end{align*}
	Therefore, $d(\xi)$ and $-d(-\xi)$ satisfy the same equation, which is uniquely solvable in a neighborhood of 0 due to the implicit function theorem, see the proof of Lemma~\ref{lem:IFT}.
	This yields $d(\xi) = -d(-\xi)$, i.e.\ $d(\xi)$ is antisymmetric.
	Since $\tau_\text{hom}(\xi) = \int_{-\infty}^\xi d(\widetilde{\xi}) \,d\widetilde{\xi}$ this proves the statement of the proposition.
\end{proof}

\subsection{Proof of the NLS-type estimate} 

We now show that the solution $v_\text{hom}$ to~\eqref{spatdyn} given by the first component of $W_\text{hom}(\xi) + \tau_\text{hom}(\xi) V_0$, that is
\begin{align*}
	v_\text{hom}(\xi) = r_0(\xi) e^{i(k_0\xi + \psi_0(\xi))} + c.c. + \curlR(\xi) + \tau_\text{hom}(\xi),
\end{align*}
with $r_0, \psi_0$ given in~\eqref{e:homsol}, is close to the soliton solution~\eqref{ahom} to the NLS equation~\eqref{NLS} with parameters
\begin{align}
	\nu_1 = 1 \text{, } \nu_2 = -s \text{ and } \gamma = s^\prime_0(c^\ast),
	\label{eq:nlsParams}
\end{align}
where we note that $s < 0$ is explicitly given by~\eqref{def:s} and $s_0'(c^\ast) > 0$ is computed in Corollary~\ref{cor:centralSpec} (thus, it holds $\nu_1\nu_2 > 0$ and $\gamma\nu_1 > 0$). We establish the estimate~\eqref{eq:NLSestimate}. Since we have $\left\|\curlR\right\|_\infty = \curlO(\varepsilon^2)$ and $\tau_\text{hom} = \curlO(\varepsilon^2)$, it suffices to prove the estimate
\begin{align*}
	\sup_{\xi \in \R} \snorm{r_0(\xi) e^{i(k_0\xi + \psi_0(\xi))} - \varepsilon A_{\text{hom},s^\prime_0(c^\ast)}(\varepsilon\xi)e^{ik_0\xi}} \leq C \varepsilon^2 \snorm{\operatorname{log}(\varepsilon)}.
\end{align*}
We split the proof of this estimate into three parts, that is, we show
\begin{align}
	\snorm{\sqrt{s_0(c)} - \varepsilon \sqrt{s^\prime_0(c^\ast)}} &\leq C \varepsilon^2, \label{eq:est1}\\
	\sup_{\xi \in \R}\snorm{\dfrac{1}{\operatorname{cosh}(\sqrt{s_0(c)}\xi)} - \dfrac{1}{\operatorname{cosh}(\varepsilon\sqrt{s^\prime_0(c^\ast)}\xi)}} &\leq C \varepsilon,  \label{eq:est2}\\
	\sup_{\xi \in \R} \snorm{f(\varepsilon \xi) \left(e^{i\psi_0(\xi)} - 1\right)} &\leq C \varepsilon \snorm{\operatorname{log}(\varepsilon)}, \label{eq:est3}
\end{align}
for $f \colon \R \rightarrow \R$ bounded such that there exists a constant $\widetilde{C} > 0$ with
\begin{align}
	\snorm{f(\xi)} \leq C e^{-\widetilde{C}\snorm{\xi}}
	\label{eq:localization}
\end{align}
for all $\xi \in \R$.

The first estimate~\eqref{eq:est1} follows directly from $s_0(c^\ast) = \curlO(\varepsilon)$, $\varepsilon = \sqrt{c-c^\ast}$ and the smoothness of $s_0$, see Corollary~\ref{cor:centralSpec}.
For the second estimate~\eqref{eq:est2} we use~\eqref{eq:est1} and the mean value theorem to bound
\begin{align*}
	\snorm{\dfrac{1}{\operatorname{cosh}(\sqrt{s_0(c)}\xi)} - \dfrac{1}{\operatorname{cosh}(\varepsilon\sqrt{s_0^\prime(c^\ast)}\xi)}} \leq C \sup_{\tilde{\xi} \in \mathcal{I}} \snorm{\operatorname{sech}(\tilde{\xi}) \operatorname{tanh}(\tilde{\xi})} \varepsilon^2 \snorm{\xi},
\end{align*}
with $\mathcal{I} = (\min\{\sqrt{s_0(c)},\varepsilon\sqrt{s_0^\prime(c^\ast)}\}\xi,\max\{\sqrt{s_0(c)},\varepsilon\sqrt{s_0^\prime(c^\ast)}\}\xi)$.
Since $\operatorname{sech}(\tilde{\xi}) \operatorname{tanh}(\tilde{\xi})$ is exponentially localized, we find
\begin{align*}
	\sup_{\xi \in \R} \sup_{\tilde{\xi} \in \mathcal{I}} \snorm{\operatorname{sech}(\tilde{\xi}) \operatorname{tanh}(\tilde{\xi})} \snorm{\xi} \leq C \varepsilon^{-1},
\end{align*}
which proves the estimate~\eqref{eq:est2}.
To obtain the final estimate~\eqref{eq:est3} we first recall that $\psi_0$ is of the form
\begin{align*}
	\psi_0(\xi) = p_0(c)\xi + \tilde{f}(\xi),
\end{align*}
with $\tilde{f} \in L^\infty(\R)$ satisfying $\|\tilde{f}\|_{L^\infty} = \curlO(\varepsilon)$, see~\eqref{e:homsol}.
Therefore, it is sufficient to prove
\begin{align*}
	\sup_{\xi \in \R} \snorm{f(\varepsilon \xi) \left(e^{ip_0(c)\xi} - 1\right)} &\leq C \varepsilon \snorm{\operatorname{log}(\varepsilon)},
\end{align*}
to obtain~\eqref{eq:est3}.
Since $f$ is assumed to be exponentially localized, there exists a $\xi_0(\varepsilon)$ with $0 < \xi_0(\varepsilon) \leq C\snorm{\operatorname{log}(\varepsilon)}$ such that $\sup_{\xi \in \R \setminus (-\xi_0,\xi_0)} |f(\xi)| \leq \varepsilon$.
Now, defining $K_\varepsilon := [-\xi_0(\varepsilon)/\varepsilon, \xi_0(\varepsilon)/\varepsilon]$ and using the mean value theorem we find
\begin{align*}
	\sup_{\xi \in K_\varepsilon} \snorm{f(\varepsilon \xi) \left(e^{ip_0(c)\xi} - 1\right)} \leq \sup_{\xi \in K_\varepsilon} \snorm{f(\varepsilon\xi) p_0(c) \xi} \leq C \varepsilon^2 \snorm{K_\varepsilon} \leq C \varepsilon \snorm{\operatorname{log}(\varepsilon)},
\end{align*}
where we used $p_0(c) = \curlO(\varepsilon^2)$, see Corollary~\ref{cor:centralSpec}.
Finally, using the definition of $K_\varepsilon$ we obtain
\begin{align*}
	\sup_{\xi \in \R \setminus K_\varepsilon} \snorm{f(\varepsilon \xi) \left(e^{ip_0(c)\xi} - 1\right)} \leq C \sup_{\xi \in \R \setminus K_\varepsilon} \snorm{f(\varepsilon\xi)} \leq C \varepsilon.
\end{align*}
This proves~\eqref{eq:est3}.
Combining~\eqref{eq:est1}--\eqref{eq:est3} then yields the NLS-type estimate~\eqref{eq:NLSestimate}, which completes the proof of Theorem~\ref{theo:mainresult}. $\hfill \Box$

\section{Generalized  (modulating) pulse solutions}

\label{sec5}

In this section we predict which other generalized modulating pulse solutions
can be constructed via spatial dynamics, center manifold reduction, and bifurcation theory
for our FPU model~\eqref{FPU}.

\subsection{KdV based generalized pulse solutions}

As explained in the introduction, the neutral eigenvalues of the spatial dynamics formulation
can be obtained by intersecting the curves $ k \mapsto  \pm\omega(k) $ and the line $ k \mapsto ck $, cf.~\eqref{disptemp}.
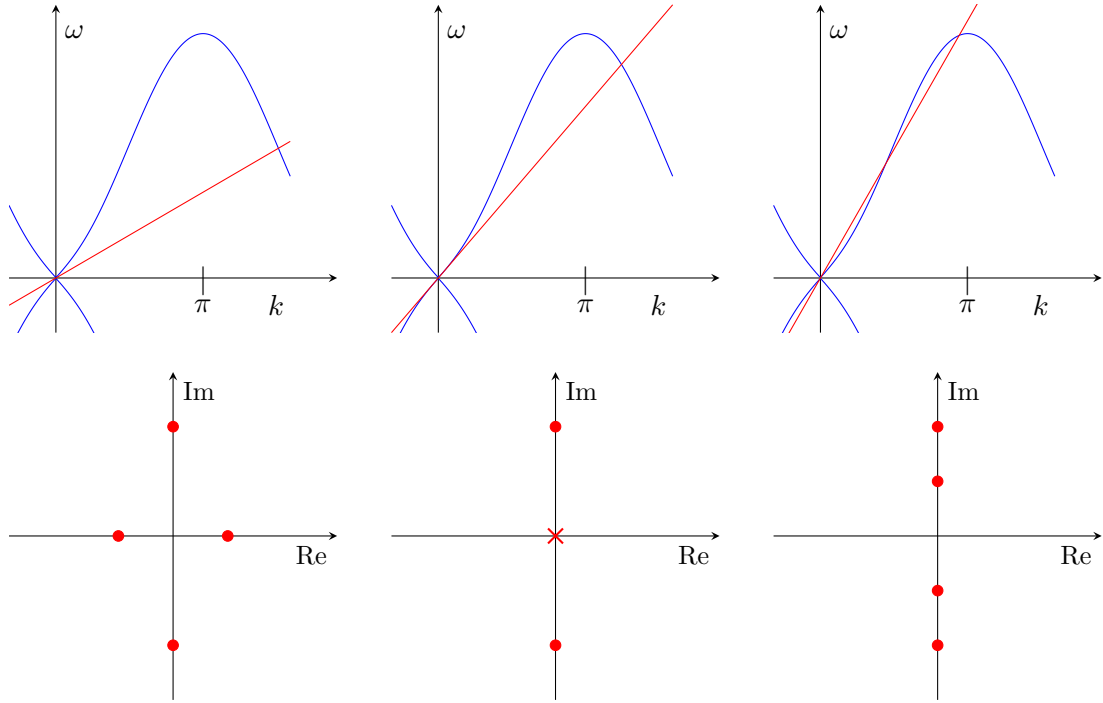
\begin{figure}
    \centering

    \begin{minipage}{0.3\textwidth}
    	\begin{tikzpicture}
    		\begin{axis}
				[
					width = 1.2\textwidth,
					height = 1.2\textwidth,
					xmin=-1, xmax=6,
					ymin=-1, ymax=5,
					axis lines=center,
					ticks=none,
    			]
 				\addplot [domain =-1:5 ,smooth, samples = 244,blue ]{sqrt(10*(1-cos(180*x/3.14))- 2*(1-cos(180*2*x/3.14)))};
 				\addplot [domain =-1:5 ,smooth, samples = 244,blue ]{-sqrt(10*(1-cos(180*x/3.14))- 2*(1-cos(180*2*x/3.14)))};
  				\addplot [domain =-1:5 ,smooth, samples = 244,red ]{0.5*x};

 				\node[black] at (axis cs:4.7,-0.5) {$k$};
 				\node[black] at (axis cs:3.14,-0.05) {$|$};
  				\node[black] at (axis cs:3.14,-0.5) {$\pi$};

  				\node[black] at (axis cs:0.4,4.5) {$\omega$};
  			\end{axis}
  		\end{tikzpicture}
  		\newline
		\newline
  		\begin{tikzpicture}
  			\begin{axis}
  				[
  					width = 1.2\textwidth,
  					height = 1.2\textwidth,
  					xmin = -3, xmax = 3,
  					ymin = -3, ymax = 3,
  					axis lines = center,
  					ticks = none,
  					xlabel = {\small$\operatorname{Re}$},
  					ylabel = {\small$\operatorname{Im}$},
  					x label style = {anchor = north east}
  				]
  				\addplot[only marks, red] table{
  					0 2
  					-1 0
  					1 0
  					0 -2
  				};
  			\end{axis}
  		\end{tikzpicture}
    \end{minipage}
    \begin{minipage}{0.3\textwidth}
    	\begin{tikzpicture}
    		\begin{axis}
				[
					width = 1.2\textwidth,
					height = 1.2\textwidth,
					xmin=-1, xmax=6,
					ymin=-1, ymax=5,
					axis lines=center,
					ticks=none,
    			]
 				\addplot [domain =-1:5 ,smooth, samples = 244,blue ]{sqrt(10*(1-cos(180*x/3.14))- 2*(1-cos(180*2*x/3.14)))};
 				\addplot [domain =-1:5 ,smooth, samples = 244,blue ]{-sqrt(10*(1-cos(180*x/3.14))- 2*(1-cos(180*2*x/3.14)))};
  				\addplot [domain =-1:5 ,smooth, samples = 244,red ]{x};

 				\node[black] at (axis cs:4.7,-0.5) {$k$};
 				\node[black] at (axis cs:3.14,-0.05) {$|$};
  				\node[black] at (axis cs:3.14,-0.5) {$\pi$};

  				\node[black] at (axis cs:0.4,4.5) {$\omega$};
  			\end{axis}
		\end{tikzpicture}
		\newline
		\newline
  		\begin{tikzpicture}
  			\begin{axis}
  				[
  					width = 1.2\textwidth,
  					height = 1.2\textwidth,
  					xmin = -3, xmax = 3,
  					ymin = -3, ymax = 3,
  					axis lines = center,
  					ticks = none,
  					xlabel = {\small$\operatorname{Re}$},
  					ylabel = {\small$\operatorname{Im}$},
  					x label style = {anchor = north east}
  				]
  				\addplot[only marks, red] table{
  					0 2
  					0 -2
  				};
  				\addplot[only marks, mark=x, red, mark size = 4pt, thick] table{
  					0 0
  				};
  			\end{axis}
  		\end{tikzpicture}
    \end{minipage}
    \begin{minipage}{0.3\textwidth}
    	\begin{tikzpicture}
    		\begin{axis}
				[
					width = 1.2\textwidth,
					height = 1.2\textwidth,
					xmin=-1, xmax=6,
					ymin=-1, ymax=5,
					axis lines=center,
					ticks=none,
    			]
 				\addplot [domain =-1:5 ,smooth, samples = 244,blue ]{sqrt(10*(1-cos(180*x/3.14))- 2*(1-cos(180*2*x/3.14)))};
 				\addplot [domain =-1:5 ,smooth, samples = 244,blue ]{-sqrt(10*(1-cos(180*x/3.14))- 2*(1-cos(180*2*x/3.14)))};
  				\addplot [domain =-1:5 ,smooth, samples = 244,red ]{1.5*x};

 				\node[black] at (axis cs:4.7,-0.5) {$k$};
 				\node[black] at (axis cs:3.14,-0.05) {$|$};
  				\node[black] at (axis cs:3.14,-0.5) {$\pi$};

  				\node[black] at (axis cs:0.4,4.5) {$\omega$};
  			\end{axis}
		\end{tikzpicture}
		\newline
		\newline
  		\begin{tikzpicture}
  			\begin{axis}
  				[
  					width = 1.2\textwidth,
  					height = 1.2\textwidth,
  					xmin = -3, xmax = 3,
  					ymin = -3, ymax = 3,
  					axis lines = center,
  					ticks = none,
  					xlabel = {\small$\operatorname{Re}$},
  					ylabel = {\small$\operatorname{Im}$},
  					x label style = {anchor = north east}
  				]
  				\addplot[only marks, red] table{
  					0 2
  					0 1
  					0 -1
  					0 -2
  				};
  			\end{axis}
  		\end{tikzpicture}
    \end{minipage}

      \caption{Intersection points of the line $ k \mapsto ck $  and
the curves $ k \mapsto \pm\omega(k) $ correspond to central
eigenvalues of the linearized spatial dynamics formulation.
Left upper panel: For subsonic wavespeeds $ c < |\omega'(0| = 1 $, except for the trivial
solution $ k=0 $,  two other  intersection points occur
and so two neutral eigenvalues are present in the spatial dynamics formulation,
cf.~left lower panel.
Middle upper panel: For the speed of sound $ c= |\omega'(0)|$  a tangent intersection occurs
leading to double zero eigenvalues. The two other intersections still exist giving again two
purely imaginary eigenvalues, cf. middle  lower panel. Right upper panel: For supersonic wave speeds $ c > |\omega'(0)|$
four non-trivial intersections occur leading to four purely imaginary
eigenvalues, cf. right lower panel.
}
    \label{fig3sec5}
\end{figure}
At the so-called speed of sound $ c = |\omega'(0)| = 1 $ two purely imaginary eigenvalues collide at the origin and split in
a positive and in a negative eigenvalue, cf Figure~\ref{fig3sec5}. Thus, for subsonic speeds, with $ c - 1 $ slightly negative,
there exists a four-dimensional invariant manifold containing a one-dimensional
unstable manifold, a one-dimensional  stable manifold, and  a two-dimensional center manifold.
For $\xi \to \pm\infty $ the solutions will converge towards small solutions on this  two-dimensional center manifold.
In general there will be no intersection of the  one-dimensional
unstable manifold the  one-dimensional  stable manifold in the  four-dimensional invariant manifold.
Hence, generically speaking, only solutions with small oscillatory tails at infinity can be found.
Since the two zero eigenvalues correspond to the wave number $ k = 0 $, the bifurcating
solutions are approximately given by associated
KdV solitary waves. Therefore, using the reversibility of the system we expect that the following theorem holds.
\begin{conjecture} \label{conj:kdv}
For every $ m \in \mathbb{N} $ and $ \tilde{c} > 0 $
there exist   $ C, \varepsilon_0  > 0 $ such that  for all $  \varepsilon \in  (0, \varepsilon_0) $
the following holds. System~\eqref{FPU} possesses generalized
moving pulse solutions of permanent form
$ q_n(t) = v(n-ct)$ with $ c = 1-  \varepsilon^2 \tilde{c} $ and smooth profile function $v\colon [-1/\varepsilon^m,1/\varepsilon^m] \to \R$
enjoying the estimate
$$
\sup_{|\xi| < 1/\varepsilon^m} | v(\xi) - h(\xi) | \leq C \varepsilon^m,
$$
where $h \colon \R \to \R$ is a smooth function satisfying
$$
\lim_{\xi \to \pm\infty} h(\xi) = 0, \qquad \sup _{\xi\in \mathbb{R} }\left| h\left( \xi\right) - \varepsilon^2 A_{sol,\tilde{c}}\left( \varepsilon \xi\right)  \right| \leq C\varepsilon^{3},
$$
and where
$  A_{sol,\tilde{c}} $ is the solitary wave of the associated KdV equation~\eqref{kdv} with speed $ \tilde{c} > 0 $.
\end{conjecture}

We note that it was already mentioned in~\cite{TRUV} that the slightly subsonic solutions in Conjecture~\ref{conj:kdv} are expected to be approximated through the KdV equation.

In case that the spatial dynamics formulation can be written as Hamiltonian system the  two-dimensional center manifold is filled with periodic solutions due to Lyapunov's subcenter theorem and $v(\xi)$ exists for all $\xi \in \R$.
In general, the solutions then only exist on spatial scales which are exponentially large with respect to the
amplitude of the bifurcating solutions~\cite{IoossJames05}.

\subsection{NLS based generalized modulating pulse solutions}

In contrast to the moving modulating pulse solutions constructed in this paper, we now consider moving modulated pulse solutions
$$
q_n(t) = v_{mp}(\xi,p) =  v_{mp}(x-c_gt ,k_0(x-c_p t)) ,
$$
with $ c_g \neq c_p $.
These solutions are  time-periodic  in a moving frame,
i.e.,
$$
q_n(t) = v_{mp}(\xi,p)=  v_{br}(\xi,t),
$$
with $ v_{br}(\xi,t) = v_{br}(\xi, t +T) $ for a suitably chosen $ T > 0 $.
Using the same arguments as in the introduction, the neutral eigenvalues of the associated spatial dynamics formulation
can be obtained by intersecting the curves $ k \mapsto  \pm \omega(k) $ and $ k \mapsto ck + 2 i \pi m/T $
for all $ m \in \mathbb{Z} $,
cf.~\cite{SU17book}.
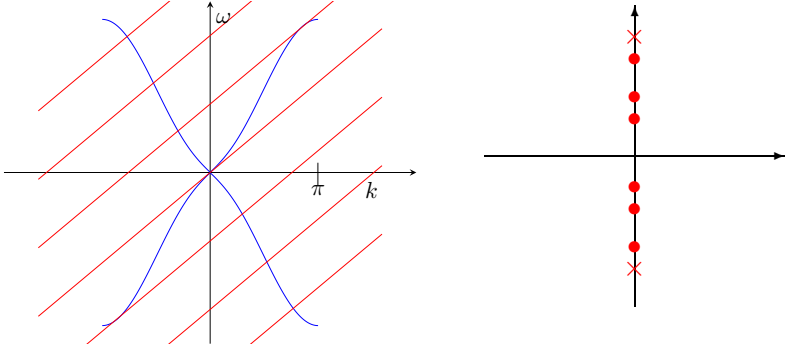
\begin{figure}
    \centering

    \begin{tikzpicture}[xscale=0.8,yscale=0.8]
\begin{axis}
[
	xmin=-6, xmax=6,
	ymin=-5, ymax=5,
	axis lines=center,
	ticks=none,
    ]
 \addplot [domain =-3.14:3.14 ,smooth, samples = 244,blue ]{sqrt(10*(1-cos(180*x/3.14))- 2*(1-cos(180*2*x/3.14)))};
 \addplot [domain =-3.14:3.14 ,smooth, samples = 244,blue ]{-sqrt(10*(1-cos(180*x/3.14))- 2*(1-cos(180*2*x/3.14)))};
  \addplot [domain =-5:5 ,smooth, samples = 244,red ]{2+0.84*x};
    \addplot [domain =-5:5 ,smooth, samples = 244,red ]{4+0.84*x};
      \addplot [domain =-5:5 ,smooth, samples = 244,red ]{0+0.84*x};
        \addplot [domain =-5:5 ,smooth, samples = 244,red ]{-2+0.84*x};
          \addplot [domain =-5:5 ,smooth, samples = 244,red ]{-4+0.84*x};
           \addplot [domain =-5:5 ,smooth, samples = 244,red ]{-6+0.84*x};
          \addplot [domain =-5:5 ,smooth, samples = 244,red ]{6+0.84*x};

 \node[black] at (axis cs:4.7,-0.5) {$k$};
 \node[black] at (axis cs:3.14,-0.05) {$|$};
  \node[black] at (axis cs:3.14,-0.5) {$\pi$};

  \node[black] at (axis cs:0.4,4.5) {$\omega$};

\end{axis}
\end{tikzpicture}\qquad
  \setlength{\unitlength}{1cm}
 \begin{picture}(7, 7)


  \put(0,2.5){\vector(1,0){4}}
   \put(2,0.5){\vector(0,1){4}}
    \put(1.85,4){\color{red}$\times$}
   \put(1.85,0.9){\color{red}$\times$}
      \put(1.9,3.7){\color{red}$\bullet$}
   \put(1.9,1.2){\color{red}$\bullet$}
      \put(1.9,3.2){\color{red}$\bullet$}
   \put(1.9,1.7){\color{red}$\bullet$}
     \put(1.9,2.9){\color{red}$\bullet$}
   \put(1.9,2){\color{red}$\bullet$}

      \end{picture}

      \caption{Intersection points of the lines $ k \mapsto ck + 2 i \pi m/T$, $m \in \mathbb{Z}$  and
the curves $ k \mapsto \pm \omega(k) $ correspond to central
eigenvalues of the linearized spatial dynamics formulation for moving modulating pulse solutions.
Tangential intersections in the left panel correspond to double eigenvalues in the right panel,
simple intersections correspond to single eigenvalues.
}
    \label{fig3sec5b}
\end{figure}
For the situation plotted in  Figure~\ref{fig3sec5b}, the discussion of the reduced system is similar
to the discussion in the last subsection, except that for $ c - \omega'(k_0) $ slightly positive
the dimensions of the stable and unstable manifolds
are now two and of the center manifold is six. Since $ k_0 \neq 0 $
the bifurcating
solutions are approximately given by associated
NLS solitary waves.
Therefore, using the reversibility of the problem, we expect  that the following theorem holds.
\begin{conjecture} Let $\nu_1, \nu_ 2 >0$, and let $m \in \mathbb{N}$. For every $ \gamma  \in \mathbb{R} $ with  $  \nu_1 \gamma > 0 $ and all $ k_0 > 0 $
there exist constants $ \varepsilon_0,c_0, C > 0$ such that  for all $  \varepsilon \in  (0, \varepsilon_0) $
the following holds. For all $(c_p,c_g) \in \R^2$ satisfying
$$\left|c_g-\frac{\omega'(k_0)}{k_0}\right|, \left|c_p - \frac{\omega(k_0)}{k_0}\right| \leq c_0,$$
system~\eqref{FPU} possesses generalized
moving  modulating pulse solutions
$ q_n(t) = v(n-c_gt,k_0(x-c_p t))$ with smooth profile function $v \colon \R \times \R/2\pi\mathbb{Z} \to \R$ enjoying the estimate
$$
\sup_{|\xi| < 1/\varepsilon^m} | v(\xi,p) - h(\xi,p) | \leq C \varepsilon^m,
$$
where $h \colon \R \times \R/2\pi\mathbb{Z} \to \R$ is a smooth function satisfying
$$
\lim_{\xi \to \pm\infty} h(\xi,p) = 0, \qquad \sup _{\xi\in \mathbb{R} }\left| h\left( \xi,p\right) -\left( \varepsilon A_{hom,\gamma}\left( \varepsilon \xi\right) e^{ip} + c.c. \right) \right| \leq C\varepsilon^{2},
$$
and where $ A_{hom,\gamma} $ is the time-periodic solution to the NLS equation~\eqref{NLS} introduced in~\eqref{ahom}.
\end{conjecture}

\bibliographystyle{abbrv}

\bibliography{fpumodpuls}

\end{document}